\documentclass{amsart}
\usepackage[in]{fullpage}
\usepackage{enumerate}

\usepackage{amsmath,amssymb,amsthm,framed,graphicx,latexsym,youngtab,epstopdf,booktabs}
\usepackage{fourier}

\usepackage{color}
\definecolor{Chocolat}{rgb}{0.36, 0.2, 0.09}
\definecolor{BleuTresFonce}{rgb}{0.215, 0.215, 0.36}
\definecolor{EgyptianBlue}{rgb}{0.06, 0.2, 0.65}

\usepackage[colorlinks,final]{hyperref}
\hypersetup{citecolor=BleuTresFonce, urlcolor=EgyptianBlue, linkcolor=Chocolat}

\def\calA{\mathcal{A}}

\def\calC{\mathcal{C}}

\def\calG{\mathcal{G}}

\def\calJ{\mathcal{J}}

\def\calN{\mathcal{N}}
\def\calO{\mathcal{O}}
\def\calP{\mathcal{P}}

\def\calR{\mathcal{R}}

\def\calT{\mathcal{T}}

\DeclareMathOperator{\rank}{rank}

\DeclareMathOperator{\Com}{Com}
\DeclareMathOperator{\Lie}{Lie}

\DeclareMathOperator{\Poisson}{Poisson}
\DeclareMathOperator{\Leib}{Leib}
\DeclareMathOperator{\Zinb}{Zinb}

\renewcommand{\arg}{\ast}

\allowdisplaybreaks

\theoremstyle{plain}
 \newtheorem{theorem}{Theorem}[section]
 \newtheorem{lemma}{Lemma}[section]
 \newtheorem{proposition}{Proposition}[section]

 \newtheorem*{maintheorem}{Main Theorem}
 \newtheorem*{niltheorem}{Nilpotency Theorem}

\theoremstyle{definition}
\newtheorem{definition}{Definition}[section]
\newtheorem{notation}{Notation}[section]
\newtheorem{example}{Example}[section]
\newtheorem{remark}{Remark}[section]
\newtheorem{algorithm}{Algorithm}[section]
\newtheorem{problem}{Problem}[section]

\numberwithin{equation}{section}

\title{Classification of regular parametrized one-relation operads}

\author{Murray Bremner}

\address{Department of Mathematics and Statistics, University of Saskatchewan, 106 Wiggins Road, S7N 5E6, Saskatoon, Canada}
\email{bremner@math.usask.ca}

\author{Vladimir Dotsenko}
\address{School of Mathematics, Trinity College Dublin, College Green, Dublin 2, Ireland and Departamento de Matem\'aticas, CINVESTAV-IPN, Av. Instituto Polit\'ecnico Nacional 2508, Col. San Pedro Zacatenco, M\'exico, D.F., CP 07360, Mexico}
\email{vdots@maths.tcd.ie}

\thanks{The collaboration of the authors on this project was partially supported by the first author's
	Discovery Grant from NSERC, the Natural Sciences and Engineering Research Council of Canada.}
\dedicatory{In memoriam JLL (1946--2012)}



\begin{document}

\title{Classification of regular parametrized one-relation operads}

\begin{abstract}
Jean-Louis Loday introduced a class of symmetric operads generated by one bilinear operation subject to one
relation making each left-normed product of three elements equal to a linear combination 
of right-normed products: 
 \[
(a_1a_2)a_3=\sum_{\sigma\in S_3}x_\sigma\, a_{\sigma(1)}(a_{\sigma(2)}a_{\sigma(3)})\ ;
 \]
such an operad is called a parametrized one-relation operad. For a particular choice of parameters $\{x_\sigma\}$,
this operad is said to be regular if each of its components is the regular representation of the symmetric group; equivalently, 
the corresponding free algebra on a vector space $V$ is, as a graded vector space, isomorphic to the tensor 
algebra of~$V$. We classify, over an algebraically closed field of characteristic zero, all regular parametrized one-relation operads. 
In fact, we prove that each such operad is isomorphic to one of the following five operads: the left-nilpotent operad 
defined by the relation $((a_1a_2)a_3)=0$, the associative operad, the Leibniz operad, the dual Leibniz (Zinbiel) operad, and the Poisson operad.
Our computational methods combine linear algebra over polynomial rings, representation theory of the symmetric group, and 
Gr\"obner bases for determinantal ideals and their radicals.
\end{abstract}
\maketitle

\section{Introduction}\label{sec:intro}

Jean-Louis Loday introduced the class of operads which he called parametrized one-relation operads. Each of these operads is 
generated by one binary operation satisfying one ternary relation which states that every monomial of the form $(a_1a_2)a_3$ can be rewritten as a linear combination of
permutations of the monomial $a_1(a_2a_3)$.
This can be regarded as a natural generalization of associativity, since it says that in each product of three
arguments we can reassociate parentheses to the right; the new feature is that we permit permutations of the
arguments.

\begin{definition}
	An operad $\calO$ generated by one bilinear operation $a_1,a_2\mapsto (a_1 a_2)$ is called a \emph{parametrized
		one-relation operad}, if its ideal of relations is generated by a single relation of the following form, called
	the \emph{LR relation} (LR for ``left-to-right'', since it allows us to re-associate parentheses in
	products of three elements from the left to the right):
\begin{equation}
	\label{LRrelation}
	\tag{LR}
 (a_1a_2)a_3=
 	x_1 a_1(a_2a_3) +
	x_2 a_1(a_3a_2) +
	x_3 a_2(a_1a_3) + 
	x_4 a_2(a_3a_1) +
	x_5 a_3(a_1a_2) +
	x_6 a_3(a_2a_1).
\end{equation}
\end{definition}

\begin{example}
	\label{fiveexamples}
	The most familiar examples of parametrized one-relation operads are the following five special cases:
	\begin{itemize}
	\item $(a_1a_2)a_3 = 0$ [left-nilpotent], 
	\item $(a_1a_2)a_3 = a_1(a_2a_3)$ [associative] 
	\item $(a_1a_2)a_3 = a_1(a_2a_3) - a_2(a_1a_3)$ [left Leibniz]
	\item $(a_1a_2)a_3 = a_1(a_2a_3) + a_1(a_3a_2)$ [right Zinbiel]
	\item $(a_1a_2)a_3 = a_1(a_2a_3) + \tfrac13 \bigl[ a_1(a_3a_2) - a_2(a_1a_3) + a_2(a_3a_1) - a_3(a_1a_2)\bigr]$ [Poisson]
	\end{itemize}
	The last identity defines the one-operation presentation of the Poisson operad discovered by Livernet and Loday~\cite{LL}. The usual definition of Poisson algebras is obtained by polarization~\cite{MR1}.
\end{example}

\begin{notation}
	For given coefficients $\mathbf{x}=[x_i]$ in Relation~\eqref{LRrelation}, we write:
	\begin{itemize}
		\item
		$\calO_\mathbf{x}$ for the quadratic symmetric binary operad defined by that relation,
		\item
		$\calO_\mathbf{x}(n)$ for the arity $n$ component of that operad (viewed as a right $S_n$-module),
		\item
		$\calO_\mathbf{x}(V)$ for the free $\calO_\mathbf{x}$-algebra generated by the vector space $V$,
	\end{itemize}
	See \S\ref{sectionpreliminaries} for a brief review of the theory of algebraic operads.
\end{notation}

Not much is known about parametrized one-relation operads in general. One natural question asked by Loday in \cite{Z} was to
determine the values of parameters for which operad $\calO_\mathbf{x}$ is Koszul. The five examples above all are Koszul,
and they have one more common feature: all components of each of these operads are regular representations of the corresponding
symmetric groups (it is obvious for the first of them, and is well known for the others, see \cite{LV}). This observation naturally leads to an attempt to search for other examples of Koszul operads among the operads
satisfying the same property.

\begin{definition}
	We say that the vector of coefficients $\mathbf{x}=[x_i]$ in Relation \eqref{LRrelation}
	is \emph{regular} if the following equivalent conditions hold:
	\begin{itemize}
		\item
		For each finite-dimensional vector space $V$, the free algebra $\calO_\mathbf{x}(V)$ is isomorphic as a graded vector space
		(\emph{not} as a graded algebra) to the non-unital tensor algebra $\overline{T}(V)$.
		\item
		For all $n \ge 1$, the $S_n$-module $\calO_\mathbf{x}(n)$ is isomorphic to the regular
		module $\mathbb{F} S_n$.
	\end{itemize}
\end{definition}

\begin{remark}
	It is often the case that the term ``regular'' is used to describe symmetric operads obtained from nonsymmetric operads by symmetrization. We choose to break that tradition and use this more general notion that includes symmetrizations of nonsymmetric operads but is wider (for instance, the operads $\Leib$ and $\Zinb$ are not symmetrizations of nonsymmetric operads): the class of operads whose free algebras have the tensor algebras as underlying vector spaces is very natural, and the term ``regular'' is most appropriate for that property.
\end{remark}

In this paper we give a complete classification of regular parametrized one-relation operads over an algebraically closed field $\mathbb{F}$ of characteristic $0$. The answer turns out to be wonderfully simple, however disappointing from the viewpoint of hunting for new Koszul operads: up to isomorphism, every such operad is one of those in Example \ref{fiveexamples}.  It is worth mentioning though, that for four of those operads, there is a one-parameter family of regular parametrized one-relation operads isomorphic to it.

\begin{maintheorem}[Theorem~\ref{th:isoclassification} (ii)]
	\label{maintheorem}
	Over an algebraically closed field of characteristic 0, every regular parametrized one-relation operad
	is isomorphic to one of the following five operads: the left-nilpotent operad defined by the identity
	$((a_1a_2)a_3)=0$, the associative operad, the Leibniz operad $\Leib$, the dual Leibniz (Zinbiel) operad
	$\Zinb$, and the Poisson operad.
\end{maintheorem}

It is an entertaining exercise to check that the five operads of Example \ref{fiveexamples} are pairwise nonisomorphic. One way to do that is as follows. The left-nilpotent operad, the associative operad, and the Poisson operad are easily seen to be isomorphic to their Koszul duals. The Koszul dual of the operad
$\Leib$ is isomorphic to the operad $\Zinb$; these two operads are not isomorphic because the suboperad
generated by the $S_2$-invariants of $\Zinb(2)$ is the operad $\Com$ of commutative associative algebras,
whereas in the case of $\Leib$, we have $(a_1a_2+a_2a_1)a_3=0$, which implies the identity $\{\{a_1,a_2\},\{a_3,a_4\}\}=0$
for the symmetrized product $\{a_1,a_2\}=a_1a_2+a_2a_1$. (In fact, it is possible to show that each identity satisfied by the symmetrized product follows from that identity).  The suboperads generated by the
$S_2$-invariants and $S_2$-anti-invariants of $\Poisson(2)$ are the operad $\Com$ and the operad $\Lie$ of
Lie algebras respectively.  Only the second of these claims holds for the associative operad, and neither
is true for the left-nilpotent operad.

The proof of the main theorem uses algorithms for linear algebra over polynomial rings, the
representation theory of the symmetric group, and commutative algebra, especially Gr\"obner bases for
determinantal ideals and their radicals. It is worth mentioning that in fact, our  proof of the main theorem shows that this classification result holds over a field $\mathbb{F}$ of characteristic zero where every quadratic equation has solutions (equivalently, $\mathbb{F}^\times=(\mathbb{F}^\times)^2$). The assumption on the characteristic is more fundamental: for example, the suboperad $\Com$ of  $\Poisson$ naturally splits off as a direct summand, and this implies that the corresponding $S_n$-modules are, in general, not regular in positive characteristic.

Our main technical result classifies all parametrized one-relation operads which are regular in arity $4$; it then turns out
that such operads are necessarily regular in all arities. It is an open problem to provide a theoretical proof which explains conceptually 
why this should be true. In a way, this phenomenon makes one think of Bergman's Diamond Lemma \cite{Bergman} in the context of operads
\cite{DK,LV}, however, there seem to be no obvious way to formalize that intuition. A related remark is that our results recover the family of operads 
from \cite{MR1} which interpolates between the associative and the Poisson operad; this family provides some supporting evidence for the operadic analogue of the Koszul deformation principle for quadratic algebras \cite{Dri,PP}; currently it is unknown if such an analogue exists.

\smallskip 

At a first glance, it is very natural to expect that most relations \eqref{LRrelation} define an operad
whose components are regular modules: one can say that re-association would permit rewriting every product as a combination of right-normed products \[ a_1 ( a_2 ( \cdots ( a_{n-1} a_n ) \cdots ) )\] which transform according to the regular representation. However, this strategy, when inspected more closely, exhibits many subtle phenomena: there are many ways at least to begin such rewriting, and at the same time, owing to the presence of all permutations on the right side of \eqref{LRrelation}, it is not at all clear that such a re-association process will terminate.  In fact, it turns out that the generic operad defined by \eqref{LRrelation} is as far from having regular modules as components as possible.

\begin{niltheorem}[Theorem \ref{th:niltheorem}]
	Let $\mathcal{N}$ be the set of all points $\mathbf{a}$ in the parameter space $\,\mathbb{F}^6$ for which
	the operad $\calO_\mathbf{a}$ is nilpotent of index $3$.
	Then $\mathcal{N}$ is a Zariski open subset of the parameter space $\mathbb{F}^6$.
\end{niltheorem}

In a nutshell, this follows from the fact that the Stasheff associahedron \cite{LV} of dimension $2$,
the pentagon, has the same number of vertices and edges; its vertices correspond to basis elements of
the free operad in arity $4$, and its edges are in one-to-one correspondence with the formal consequences
of one ternary relation.  Since the two numbers coincide, it is natural to expect that for a generic
relation all operations of arity $4$ will vanish.

\subsubsection*{Outline of this paper}

The pages that follow consist of the following sections.

Section \ref{sectionpreliminaries} recalls the necessary background on algebraic operads. We focus on binary quadratic operads, since it is the only type of operads that we consider.

Section \ref{sectionlinearalgebrapolynomialrings} reviews basics of linear algebra over
polynomial rings; we recall the notion of a determinantal ideal which is used to understand how the rank of a matrix with
polynomial entries depends on the parameters.

Section \ref{cubicrelationmatrixsection} introduces the cubic relation matrix $M$, square of size 120,
with entries in $\calC = \mathbb{F}[x_1,\dots,x_6]$.
This sparse matrix (over $94\%$ zeros) is the main object of study throughout the paper.
Its row module $\mathrm{Row}(M)$ over $\calC$ is the $S_4$-module of relations satisfied by the
general parametrized one-relation operad in arity $4$.
We use the algorithms from the previous section to obtain some basic information about the nullmodule of $M$:
the $\mathcal{C}$-module $\big\{ \, MH = O \mid H \in \calC^{120} \, \big\}$.
In particular, we prove the Nilpotency Theorem for parametrized one-relation operads.

Section \ref{representationsection} recalls basic concepts and methods from the representation theory of the
symmetric group, emphasizing arity $4$ and applications to polynomial identities.
This allows us to replace the single large matrix $M$ with five much smaller matrices which are much easier
to study using computational commutative algebra.

In Section \ref{proofclassificationsection}, we combine the approaches of the previous sections and prove
the main technical result, a classification of all parametrized one-relation operads for which the arity $4$ component 
is the regular module. This is done by a careful analysis of possible relations \eqref{LRrelation} by increasing number of nonzero coefficients.

In Section \ref{actualclassification}, we establish that each of the operads in the previous section is 
regular, and isomorphic to one of the five operads from Main Theorem, thus obtaining a full classification.

Section \ref{sectionirregular} outlines some further research directions and open problems.


\subsubsection*{Acknowledgements. } We thank CIMAT (Guanajuato, Mexico) for hosting
	the CIMPA Research School \emph{``Associative and Nonassociative Algebras and Dialgebras: Theory and Algorithms''} during which our collaboration began.
	The second author thanks the Department of Mathematics and Statistics at the University of Saskatchewan
	for its hospitality during his visit in July 2015.
	The first author thanks Michael Monagan and Jiaxiong Hu for numerous conversations on \texttt{Maple} programming, and
	Mark Spivakovsky and Jonathan Lee for correspondence about commutative algebra. The second author thanks Muriel Livernet for sharing her unpublished M.~Sc.~thesis \cite{Liv} written under direction of Jean-Louis Loday. The class of operads studied in this paper was a topic of interest for Jean-Louis for many years; he asked the second author the question that is resolved in the present paper in Luminy in February 2012. We dedicate this belated answer to his memory.

\section{Preliminaries on algebraic operads}
\label{sectionpreliminaries}

In this section we recall basic background information from the theory of operads. All operads in this paper
are generated by one binary operation, and we choose to keep this section within these limits. For general
definitions and further details, we refer the reader to the recent comprehensive monograph by
Loday and Vallette~\cite{LV}. For the algorithmic aspects, see \cite{BD}. 

\subsubsection*{Nonsymmetric operads}

Operads encode multilinear operations with many arguments in the same way as associative algebras encode
linear maps. The first level of abstraction is the notion of a nonsymmetric operad, where operations can be substituted into one another, but arguments of operations cannot be permuted. We may therefore choose a symbol such as $\arg$ to represent each of the $n$ arguments of a given operation: $\omega(\arg,\dots,\arg)$.
The different occurrences of $\arg$ represent different arguments, which are distinguished by their positions.

Throughout the paper, we only consider the case in which all operations are built out of one generating operation; therefore, we shall not give that operation a specific name, and write simply $(\arg \arg)$, where it is understood that every pair of parentheses contains exactly two arguments, and each of these arguments is in turn either $\arg$ or another pair of parentheses containing \dots, etc. This notation remains unambiguous if we also omit the commas separating the arguments.

\begin{definition}
	\label{arityfunction}
	The \emph{free nonsymmetric operad} $\Omega$ generated by one binary operation $(\arg \arg)$ has \emph{components} $\Omega(n)$, $n \ge 1$, where $\Omega(n)$ is spanned by the composite operations built out of $(\arg \arg)$ that have exactly $n$ arguments (in other words, it is of \emph{arity} $n$). Such an operation must have exactly $n-1$ occurrences of $(\arg \arg)$ (in other words, is of \emph{weight} $n-1$).
\end{definition}

\begin{example}\label{ex:basismonomials}
	The following \emph{balanced bracketings} form a basis of $\Omega(n)$ for $1\le n\le 4$:
 \[
	\left\{
	\begin{array}{rl}
	n & \text{monomials}
	\\
	1 &
	\ast
	\\
	2 &
	(\ast\ast)
	\\
	3 &
	(\ast(\ast\ast)), \qquad
	((\ast\ast)\ast)
	\\
	4 &
	(\ast(\ast(\ast\ast))), \qquad
	(\ast((\ast\ast)\ast)), \qquad
	((\ast\ast)(\ast\ast)), \qquad
	((\ast(\ast\ast))\ast), \qquad
	(((\ast\ast)\ast)\ast)
	\\
	\end{array}
	\right.
\]
	From now on we will omit the outermost pair of parentheses.
\end{example}

\begin{lemma}[\cite{S}]
	The dimension of $\Omega(n)$, or equivalently the number of distinct balanced bracketings using $n-1$ pairs of brackets, is equal to the Catalan number:
	\begin{equation}
	\label{catalan}
	\dim\,\Omega(n) = \frac{1}{n} \binom{2n-2}{n-1} \qquad (n \ge 1).
	\end{equation}
\end{lemma}

\medskip

As a vector space, $\Omega(n)$ is the homogeneous subspace of degree $n$ in the free nonassociative algebra with one binary operation $\omega$ and one generator $\arg$, but the collection of all components $\Omega(n)$ has a much richer structure to it, which exemplifies the simplest case in the theory of algebraic operads.

\begin{definition}
	\label{composition}
	The \emph{composition maps} $\circ_i$ in the free nonsymmetric operad $\Omega$ are defined as follows. On basis monomials $\mu \in \Omega(n)$ and $\mu' \in \Omega(n')$, the $i$-th composition $\mu \circ_i \mu' \in \Omega(n+n'-1)$ for $1 \le i \le n$, is the result of substituting $\mu'$ for the $i$-th argument $\arg$ in $\mu$. This operation extends bilinearly to any elements $\alpha\in\Omega(n)$ and $\alpha'\in\Omega(n')$.
\end{definition}

\begin{definition}
	\label{totalorder}
	We inductively define a total order $\mu \prec \mu'$ on nonsymmetric basis monomials $\mu$ and $\mu'$.
	The basis of the induction is the unique total order on the set $\{ \arg \}$ which is a basis of $\Omega(1)$.
	Consider $\mu \in \Omega(n)$ and $\mu' \in \Omega(n')$ where $n$ and $n'$ are not both equal to $1$.
	If $n< n'$ then we set $\mu \prec \mu'$.
	If $n'< n$ then we set $\mu' \prec \mu$.
	Otherwise, $n = n'$;
	write $\mu = \mu_1\mu_2$ and $\mu' = \mu'_1\mu'_2$. We have $\mu_i \in \Omega(p_i)$ for $p_i< n$ and $\mu'_i \in \Omega(p'_i)$ for $p'_i < n'$, therefore by induction
	we may assume that our total order is defined for $\mu_i$ and $\mu'_i$. If $\mu_1\ne \mu_1'$ we set $\mu\prec \mu'$ if and only if $\mu_1\prec \mu_1'$, else we set $\mu\prec \mu'$ if and only if $\mu_2\prec \mu_2'$. For example, the monomials in Example \ref{ex:basismonomials} follow this order.
\end{definition}

\subsubsection*{Symmetric operads}

Of course when one deals with actual multilinear operations, there is more structure to take into account,
namely permutations of arguments. Formalizing that leads to the notion of a symmetric operad.

\begin{definition}
	\label{symmetricdefinition}
	The \emph{free symmetric operad} $\calT$ generated by one binary operation has components
	\begin{equation}
	\label{symmetricoperad}
	\calT(n) \, = \, \Omega(n) \, \otimes \, \mathbb{F}S_n,
	\end{equation}
	where $S_n$ acts trivially on $\Omega(n)$ and $\mathbb{F}S_n$ is the right regular module.
	A basis for $\calT(n)$ consists of all simple tensors $\psi \otimes \tau$ where
	$\psi \in \Omega(n)$ is a nonsymmetric basis monomial and $\tau \in S_n$ is a permutation of the arguments.
\end{definition}

\begin{remark}
	The natural interpretation of the simple tensor $\psi \otimes \tau$ is that $\psi$ represents a certain
	bracketing (or placement of operation symbols) applied to the underlying multilinear monomial
	$a_{\tau(1)} \cdots a_{\tau(n)}$, which is the
	result of the action of $\tau$ on a decomposable tensor $a_1\otimes \cdots\otimes a_n$.
	Since this action of $S_n$ can lead quickly to a great deal of confusion, we write a few sentences to clarify it.
	Consider this left action of $S_n$ on decomposable tensors $v_1\otimes \cdots\otimes v_n$:
	\[
	\tau( v_1 \otimes \cdots \otimes v_n) = v_{\tau^{-1}(1)} \otimes \cdots \otimes v_{\tau^{-1}(n)}.
	\]
	This action moves the factor in position $i$ to position $\tau(i)$, and induces a right action on $\calT(n)$
	which has the property that its extension to the tensor product $\calT(n) \otimes_{\mathbb{F}S_n} V^{\otimes n}$
	can be conveniently interpreted as applying operations to arguments.
	In other words,
	\begin{equation}
	\label{groupaction}
	( \psi \otimes \tau ) \cdot \sigma
	=
	\psi \otimes \tau\sigma.
	\end{equation}
	The total order of Definition \ref{totalorder} extends from the nonsymmetric case to the symmetric case:
	given basis monomials $\psi \otimes \tau$ and $\psi' \otimes \tau'$, we first compare the
	bracketings $\psi, \psi'$, and if $\psi = \psi'$ then we compare the permutations $\tau, \tau'$
	in lexicographical order.
	It is straightforward to verify that the natural composition of operations in $\calT$ is equivariant with
	respect to this action of the symmetric groups.
	More concretely, one can view $\calT(n)$ as the multilinear subspace of degree $n$ in the free nonassociative algebra with one binary operation and $n$ generators $a_1,\dots,a_n$.
\end{remark}

\begin{lemma}
	\label{symmetricdimension}
	The dimension of $\calT(n)$, or equivalently the number of distinct multilinear $n$-ary  nonassociative monomials, is given by the following formula:
	\[
	\dim\,\calT(n)
	\;=\;
	\frac{1}{n} \binom{2n-2}{n-1} n!\, .
	\]
\end{lemma}

\begin{proof}
	This follows immediately from Definition \ref{arityfunction} and equations \eqref{catalan} and
	\eqref{symmetricoperad}.
\end{proof}

\begin{definition}
	By a \emph{quadratic relation} in the free symmetric operad $\calT$ we mean an element of $\calT(3)$: a (nonzero) linear combination of simple tensors $\psi \otimes \tau$ where each bracketing $\psi \in \Omega(3)$ involves two occurrences of the generator $(\arg,\arg)$. Any $S_3$-submodule $\calR \subseteq \calT(3)$ is called a \emph{module of quadratic relations}.
	To determine a module of relations $\calR$, it suffices to give a set of module generators, not a
	linear basis (which is typically much larger).
\end{definition}

\begin{remark}
	When discussing relations in an operad, the word quadratic (and similarly cubic, quartic, etc.) refers to the weight $k-1$, not to the arity $k$.
	In particular, the quadratic relations are of arity three.
\end{remark}

\begin{definition}
	An \emph{(operad) ideal} $\calJ$in the free symmetric operad $\calT$ is a family of $S_n$-submodules $\calJ(n) \subseteq\calT(n)$, where $n \ge 1$, which is closed
	under composition with arbitrary elements of $\calT$.
	
	The \emph{quotient operad} $\calT \,/ \calJ$ has components $(\calT \,/ \calJ)(n) = \calT(n) / \calJ(n)$ with the natural induced compositions.
	
	The ideal $\calJ \subseteq \calT$ \emph{generated} by a subset $\Gamma \subseteq \calT$
	is the intersection of all the ideals containing $\Gamma$. Notation: $\calJ=(\Gamma)$.
\end{definition}

\begin{definition}
	\label{definitionconsequences}
	Consider the operad ideal $\calJ = (\rho)$ generated by one quadratic relation $\rho \in \calT(3)$.
	The $S_3$-module of quadratic relations is $\calJ(3) = \rho \cdot \mathbb{F}S_3$,
	the right $S_3$-module generated by $\rho$.
	We regard $\rho = \rho(a_1,a_2,a_3)$ as an operation with three arguments.
	The component $\calT(4)$ contains three compositions $\rho \circ_i \omega$ and two compositions $\omega \circ_j \rho$ which vanish in $(\calT/(\rho))(4)$:
	\begin{equation}
	\label{generalconsequences}
	\rho((a_1a_2),a_3,a_4),\quad
	\rho(a_1,(a_2a_3),a_4),\quad
	\rho(a_1,a_2,(a_3a_4)),\quad
	\rho(a_1,a_2,a_3)a_4,\quad
	a_1\rho(a_2,a_3,a_4).
	\end{equation}
	We call relations \eqref{generalconsequences} the \emph{cubic consequences} of the quadratic relation $\rho$. These five relations generate the $S_4$-module $\calJ(4) \subseteq \calT(4)$. We can inductively repeat this generation of consequences into higher arities to compute every $S_n$-module
	in the ideal $( \rho )$, but we will only require the cubic case.
\end{definition}

\begin{definition}
We say that an operad $\calP = \calT\,/\calJ$ is \emph{nilpotent} if there exists $k_0 \ge 0$ such that $\calP(k+1) = \{0\}$ for all $k\ge k_0$.
If $k_0$ is the least nonnegative integer satisfying this condition, then we say that $\calP$ is nilpotent of \emph{index}~$k_0$. 
 (This way, nilpotency of index $k_0$ means that all operations made of $k_0$ or more copies of the generating operation vanish).
\end{definition}

Clearly $\calP(k) = \{0\}$ if and only if $\calJ(k) = \calT(k)$. Hence $\calP$ is nilpotent of index $k_0$ if and only if $\calJ(k_0) \ne \calT(k_0)$ and
$\calJ(k) = \calT(k)$ for all $k>k_0$. Compositions of elements of $\calT(j)$ with
the generating operation produce all of $\calT(j+1)$, so to check nilpotency it is enough to check that $\calP(k+1)=0$ just for $k=k_0$, and not for all $k\ge k_0$.

\begin{example}[\cite{MR3}]
	\label{antiassociativeexample}
	The simplest example of a nilpotent operad is the \emph{anti-associative} operad $\calA^+$ generated
	by one binary operation satisfying the relation 
	 \[(a_1a_2)a_3 + a_1(a_2a_3) = 0 ; \] this relation introduces a
	sign change every time we reassociate a product of three factors.
	This relation is the special case with parameters
	$[-1, 0, 0, 0, 0, 0]$ of Relation \eqref{LRrelation}; hence $\calA^+$ is a parametrized one-relation operad.
	It is easy to show that $\calA^+$ is nilpotent of index $3$. Indeed, we note that the defining relation of our operad can be applied as a rewriting rule to the product $((a_1a_2)a_3)a_4$ in two different ways, by rewriting $(a_1a_2)a_3$ first, obtaining
	\[
	((a_1a_2)a_3)a_4 = -(a_1(a_2a_3))a_4 = a_1((a_2a_3)a_4) = -a_1(a_2(a_3a_4)) ,
	\]
	or by setting $b=(a_1a_2)$ and rewriting $(ba_3)a_4$ first, obtaining.
	\[
	((a_1a_2)a_3)a_4 = -(a_1a_2)(a_3a_4) = a_1(a_2(a_3a_4)).
	\]
	(This should remind the reader of computing an S-polynomial when calculating a Gr\"obner basis).
	We conclude that $a_1(a_2(a_3a_4)) = 0$.
	Since all five basis compositions \eqref{generalconsequences} appear along the way, all of them are zero.
	Hence $\calA^+(4) = \{0\}$, and the operad $\calA^+$ is nilpotent. 
\end{example}

\subsubsection*{Matrix condition for regularity}

Relation \eqref{LRrelation} is a special case of the following general binary quadratic relation, first considered in \cite{Liv}:
\begin{equation}
\label{generalrelation}
R(a_1,a_2,a_3) =
\sum_{\tau \in S_3} w_\tau \, \big( a_{\tau(1)}a_{\tau(2)} \big) \, a_{\tau(3)}
+
\sum_{\tau \in S_3} y_\tau \, a_{\tau(1)}\big( a_{\tau(2)} \, a_{\tau(3)}\big),
\end{equation}
where $w_\tau, y_\tau \in \mathbb{F}$. The $S_3$-submodule generated by $R$ is the module $(R) \cap \calT(3)$ of quadratic relations.
If $H \subseteq S_3$ is the (normal) subgroup fixing $R$ then $(R) \cap \calT(3) \cong \mathbb{F}(S_3/H)$
and so $\dim(R) \le 6$, with equality if and only if only the identity permutation fixes $R$.
The larger $\dim(R)$, the smaller $H$: dimension and symmetry are inversely related.
For us the important case is $\dim(R) = 6$: thus $R$ generates an $S_3$-module
isomorphic to $\mathbb{F}S_3$.
Relation \eqref{LRrelation} satisfies this condition.
We shall return to this general relation \eqref{generalrelation} in Section \ref{representationsection}
where it will serve as a toy example for representation-theoretic method.

We write out relation \eqref{generalrelation} term by term, replacing the
permutation subscripts by integers, using the lex order in $S_3$.
The relation $R$ then has the following form:
\begin{align}
\label{relation2}
w_1 (a_1a_2)a_3 + w_2 (a_1a_3)a_2 + w_3 (a_2a_1)a_3 + w_4 (a_2a_3)a_1 + w_5 (a_3a_1)a_2 + w_6 (a_3a_2)a_1
\\
\notag \; +
y_1 a_1(a_2a_3) + y_2 a_1(a_3a_2) + y_3 a_2(a_1a_3) + y_4 a_2(a_3a_1) + y_5 a_3(a_1a_2) + y_6 a_3(a_2a_1).
\end{align}
For a relation $R$ of the form \eqref{LRrelation} we have $w_1 = 1$ and $w_2 = \cdots = w_6 = 0$.
Let $[ \, W \mid Y \, ]$ be the matrix whose rows are the coefficient vectors obtained by applying every
$\sigma \in S_3$ to $R$:
\[
R \cdot \sigma
\;=\;
\sum_{\tau \in S_3}
w_\tau \, \big( a_{\sigma\tau(1)}a_{\sigma\tau(2)} \big) \, a_{\sigma\tau(3)}
+
\sum_{\tau \in S_3}
y_\tau \, \big( a_{\sigma\tau(1)}a_{\sigma\tau(2)} \big) \, a_{\sigma\tau(3)},
\qquad
\sigma\in S_3.
\]
Working this out explicitly, where the columns correspond to the basis monomials in the order of
\eqref{relation2}, we obtain a matrix where the pattern of subscripts matches that of the celebrated
Dedekind--Frobenius determinant for $S_3$:
\begin{equation}
\label{rowspace1}
\left[ \begin{array}{c|c} W & Y \end{array} \right]
\;=\;
\left[
\begin{array}{cccccc|cccccc}
w_1 & w_2 & w_3 & w_4 & w_5 & w_6 & y_1 & y_2 & y_3 & y_4 & y_5 & y_6 \\
w_2 & w_1 & w_5 & w_6 & w_3 & w_4 & y_2 & y_1 & y_5 & y_6 & y_3 & y_4 \\
w_3 & w_4 & w_1 & w_2 & w_6 & w_5 & y_3 & y_4 & y_1 & y_2 & y_6 & y_5 \\
w_5 & w_6 & w_2 & w_1 & w_4 & w_3 & y_5 & y_6 & y_2 & y_1 & y_4 & y_3 \\
w_4 & w_3 & w_6 & w_5 & w_1 & w_2 & y_4 & y_3 & y_6 & y_5 & y_1 & y_2 \\
w_6 & w_5 & w_4 & w_3 & w_2 & w_1 & y_6 & y_5 & y_4 & y_3 & y_2 & y_1
\end{array}
\right]
\end{equation}

\begin{lemma}
\label{leftminor}
Suppose that for the given $6\times 6$ matrices $W$ and $Y$ with coefficients in $\mathbb{F}$ the rows of the matrix $[ \, W \mid Y \, ]$ form a single orbit for the right action of $S_3$, as in \eqref{rowspace1} above. The subspace they generate contains a relation of the type \eqref{LRrelation} if and only if $W$ is invertible.
\end{lemma}

\begin{proof}
Note that every matrix representing the orbit of a relation of the type \eqref{LRrelation} is a matrix of the form
\begin{equation}
\label{rowspace2}
\left[
\begin{array}{cccccc|cccccc}
1   & 0   & 0   & 0   & 0   & 0   & -x_1 & -x_2 & -x_3 & -x_4 & -x_5 & -x_6 \\
0   & 1   & 0   & 0   & 0   & 0   & -x_2 & -x_1 & -x_5 & -x_6 & -x_3 & -x_4 \\
0   & 0   & 1   & 0   & 0   & 0   & -x_3 & -x_4 & -x_1 & -x_2 & -x_6 & -x_5 \\
0   & 0   & 0   & 1   & 0   & 0   & -x_5 & -x_6 & -x_2 & -x_1 & -x_4 & -x_3 \\
0   & 0   & 0   & 0   & 1   & 0   & -x_4 & -x_3 & -x_6 & -x_5 & -x_1 & -x_2 \\
0   & 0   & 0   & 0   & 0   & 1   & -x_6 & -x_5 & -x_4 & -x_3 & -x_2 & -x_1
\end{array}
\right] ,
\end{equation}
and this matrix is in row canonical form (RCF). For any matrix $[ \, W \mid Y \, ]$, its RCF is a matrix of the form $[ \, I \mid Z \, ]$ if and only if $W$ is invertible, and in this case $Z=W^{-1}Y$. Finally, in \eqref{rowspace1} the matrices $W$ and $Y$ are, respectively, the matrices representing the action of $\sum_{\sigma}w_\sigma\sigma$ and $\sum_{\sigma}y_\sigma\sigma$ on the right regular module, and thus so is $W^{-1}Y$, justifying the same Dedekind--Frobenius determinant pattern of matrix elements in $W^{-1}Y$. 
\end{proof}

\subsubsection*{Koszul duality}

The theory of Koszul duality for operads, due to Ginzburg and Kapranov \cite{GK}, associates to a quadratic operad $\calP$ another quadratic operad $\calP^!$, its Koszul dual. In the case when $\calP$ satisfies some good homological properties (such operads are called Koszul operads), the Koszul dual operad can be used to control deformation theory of $\calP$-algebras. (Familiar examples are given by deformation complexes of associative algebras and Lie algebras). For an operad generated by a binary product, the operad $\calP^!$ admits a very economic description that we recall here, referring the reader to \cite{LV} for general definitions and results on Koszul duality, as well as further motivation.

\begin{proposition}[\cite{LV}]\label{koszuldualoperad}
	Suppose that $\calP\cong\calT/(\calR)$ is a quotient operad of $\calT$ by some module of quadratic relations $\calR$.
	We define a scalar product on $\calT(3)$ as follows:
	\begin{equation}
	\label{scalarproduct}
	( \psi_1, \psi_1 ) =  1,
	\; 
	( \psi_2, \psi_2 ) = -1,
	\; 
	( \psi_1, \psi_2 ) =  0,
	\; 
	\text{ where }
	\; 
	\psi_1 = (\arg\arg)\arg,
	\; 
	\psi_2 = \arg(\arg\arg).
	\end{equation}
	This can be extended to an $S_3$-invariant scalar product on $\calT(3)$ by the formula
	\begin{equation}
	\label{scalarproduct2}
	( \psi_i \otimes \tau_j, \psi_k \otimes \tau_\ell )
	=
	(\psi_i,\psi_k) \, \delta_{j\ell} \varepsilon( \tau_j ) ,
	\end{equation}
	where $\varepsilon\colon S_3 \to \{ \pm 1 \}$ is the sign, and $\delta_{j\ell}$ is the Kronecker symbol. We write $\calR^\bot$ for its orthogonal complement with respect to \eqref{scalarproduct}. The \emph{Koszul dual operad} $\calP^!$ is the quotient operad $\calT/(\calR^\bot)$.
\end{proposition}

\begin{lemma}[{\cite{Liv},\cite[Prop.~7.6.8]{LV}}]
	\label{duallemma}
	The Koszul dual operad $\calP^!$ of any parametrized one-relation operad $\calP$ is isomorphic to a parameterized one-relation operad; if the operad $\calP$ is defined by Relation \eqref{LRrelation}, the Koszul dual operad is isomorphic to the operad defined by the relation
	\[
	(a_1a_2)a_3 = x_1 a_1(a_2a_3) - x_3 a_1(a_3a_2) - x_2 a_2(a_1a_3) + x_4 a_2(a_3a_1) + x_5 a_3(a_1a_2) - x_6 a_3(a_2a_1) .   
    \]
	In plain words, to obtain $S$, we switch and negate coefficients 2 and 3, and negate coefficient 6.
\end{lemma}

\begin{proof}
	We start from matrix \eqref{rowspace2} whose row space is the module $\calR$ of quadratic relations. By Proposition \ref{koszuldualoperad}, the computation of $\calR^\bot$ is reduced to
	the computation of the nullspace of a modified matrix: we multiply columns $7$--$12$ by $-1$ according to \eqref{scalarproduct}, and then multiply the columns
	$2,3,6,8,9,12$ with odd permutations by $-1$ according to \eqref{scalarproduct2}.
	We compute the RCF; for this we simply multiply the rows with odd permutations by $-1$:
	\begin{equation}
	\label{rowspace3}
	\left[
	\begin{array}{ccccccrrrrrr}
	1 & 0 & 0 & 0 & 0 & 0 & -x_1 &  x_2 &  x_3 & -x_4 & -x_5 &  x_6 \\
	0 & 1 & 0 & 0 & 0 & 0 &  x_2 & -x_1 & -x_5 &  x_6 &  x_3 & -x_4 \\
	0 & 0 & 1 & 0 & 0 & 0 &  x_3 & -x_4 & -x_1 &  x_2 &  x_6 & -x_5 \\
	0 & 0 & 0 & 1 & 0 & 0 & -x_5 &  x_6 &  x_2 & -x_1 & -x_4 &  x_3 \\
	0 & 0 & 0 & 0 & 1 & 0 & -x_4 &  x_3 &  x_6 & -x_5 & -x_1 &  x_2 \\
	0 & 0 & 0 & 0 & 0 & 1 &  x_6 & -x_5 & -x_4 &  x_3 &  x_2 & -x_1
	\end{array}
	\right]
	\end{equation}
	We compute the standard basis for the nullspace of \eqref{rowspace3} by setting free variables to unit vectors
	and solving for leading variables.
	We obtain another matrix whose row space is the nullspace of \eqref{rowspace3}; this is
	the module $(\calR^\bot)$.
	However, this matrix has the form $[ \, X \mid I_6 \, ]$: it is not in row canonical form.
	Computing the RCF of this matrix requires dividing by polynomials in the $x_i$. However, this can be avoided by passing to the isomorphic operad for the opposite algebras, which interchanges $\psi_1$ and $\psi_2$, putting the columns back
	into the original order of the monomials, and then computing the RCF.
	We obtain the following result:
	\begin{equation}
	\label{dualmodule}
	\left[
	\begin{array}{ccccccrrrrrr}
	1 & 0 & 0 & 0 & 0 & 0 &  - x_1 & x_3 & x_2 &  -x_4 &  -x_5 & x_6 \\
	0 & 1 & 0 & 0 & 0 & 0 & x_3 &  -x_1 &  -x_5 & x_6 & x_2 &  -x_4 \\
	0 & 0 & 1 & 0 & 0 & 0 & x_2 &  -x_4 &  -x_1 & x_3 & x_6 &  -x_5 \\
	0 & 0 & 0 & 1 & 0 & 0 &  -x_5 & x_6 & x_3 &  -x_1 &  -x_4 & x_2 \\
	0 & 0 & 0 & 0 & 1 & 0 &  -x_4 & x_2 & x_6 &  -x_5 &  -x_1 & x_3 \\
	0 & 0 & 0 & 0 & 0 & 1 & x_6 &  -x_5 &  -x_4 & x_2 & x_3 &  -x_1
	\end{array}
	\right]
	\end{equation}
	From the first row of \eqref{dualmodule}, we easily read off the coefficients of $S$.
\end{proof}

\section{Linear algebra over polynomial rings}
\label{sectionlinearalgebrapolynomialrings}

Over a field $\mathbb{F}$, to determine whether two $m \times n$ matrices $A$ and $B$ belong to the same orbit
under the left action of $GL_m(\mathbb{F})$, we compute the row canonical forms $\mathrm{RCF}(A)$ and
$\mathrm{RCF}(B)$ and check whether they are equal.
Similarly, for the left-right action of $GL_m(\mathbb{F}) \times GL_n(\mathbb{F})$, we compute
$\mathrm{Smith}(A)$ and $\mathrm{Smith}(B)$.

Over a Euclidean domain, in particular, the ring $\mathbb{F}[x]$ of polynomials in one variable $x$ over a
field $\mathbb{F}$, a modification of Gaussian elimination gives the desired result, since the coordinate ring
is a PID and we can implement the Euclidean algorithm for GCDs using row (or column) operations.
The analogue of the RCF in this case is called the Hermite normal form (HNF).

Once we go beyond Euclidean domains, these computations become much more difficult, for two main reasons:
we can no longer compute GCDs using row operations, and it wouldn't help even if we could, since the
coordinate ring is no longer a PID.
In this setting, the existence of a normal form which determines when two matrices belong to the same orbit
remains an open problem.
We can nonetheless obtain some useful information about a multivariate polynomial matrix by elementary methods.

We consider the problem of computing the rank of an $m \times n$ matrix $A$ with entries in
the ring $\mathbb{F}[x_1,\dots,x_p]$ of polynomials in $p \ge 2$ variables (or parameters) over $\mathbb{F}$.
In one sense, the rank of such a matrix is its rank when regarded as a matrix over the field
$\mathbb{F}(x_1,\dots,x_p)$ of rational functions: since the coordinate ring is now a field again, we can use
Gaussian elimination.
However, crucial information is lost, since we are implicitly assuming that none of the denominators that
arise in the matrix entries during this calculation ever become 0.
Another definition of the rank of the matrix $A$ is as follows.

\begin{definition}
	\label{definitionsubrank}
	Let $A$ be an $m \times n$ matrix over $\mathbb{F}[x_1,\dots,x_p]$ regarded as a parametrized family of
	matrices over $\mathbb{F}$.
	We define the function $A| \colon \mathbb{F}^p \to \mathrm{Mat}_{mn}(\mathbb{F})$:
	for $a_1,\dots,a_p \in \mathbb{F}$ the matrix $A|(a_1,\dots,a_p)$ is obtained from $A$ by setting
	$x_i = a_i$ for $i = 1,\dots,p$.
	Composing $A|$ with the rank on $\mathrm{Mat}_{mn}(\mathbb{F})$ gives the \emph{substitution rank}
	function:
	\[
	\mathrm{subrank}_A = \mathrm{rank} \circ A| \colon \mathbb{F}^p \to \{ \, 0, 1, 2, \dots, \min(m,n) \, \}.
	\]
	The inverse images of the ranks $0 \le r \le \min(m,n)$ define the \emph{inverse rank} function:
	\[
	\mathrm{invrank}_A( r )
	=
	\big\{ \, (a_1,\dots,a_p) \in \mathbb{F}^p \mid \mathrm{subrank}_A(a_1,\dots,a_p) = r \, \big\}.
	\]
We define the \emph{minimal rank} $r_\mathrm{min}$ as follows:
	\[
	r_\mathrm{min}(A) =
	\min \big\{ \, r \mid 0 \le r \le \min(m,n), \; \mathrm{invrank}_A( r ) \ne \emptyset \, \big\}.
	\]
\end{definition}

The following very simple result will be useful to us later.

\begin{proposition}
	\label{fullranklemma}
	Let $A$ be an $m \times n$ matrix over $\mathbb{F}[x_1,\dots,x_p]$.
	If there exist elements $a_1,\dots,a_p \in \mathbb{F}$ such that the matrix
	$A|(a_1,\dots,a_p) \in \mathrm{Mat}_{mn}(\mathbb{F})$ has full rank $r = \min(m,n)$,
	then $A$ has full rank over the field $\mathbb{F}(x_1,\dots,x_p)$ of rational functions.
\end{proposition}

\begin{proof}
	It is well known that the rank of an $m \times n$ matrix over $\mathbb{F}$ is $r$ if and only if two conditions hold:
	\begin{itemize}
		\item
		At least one $r \times r$ minor is not 0.
		\item
		Every $(r+1) \times (r+1)$ minor is 0.
	\end{itemize}
	Therefore, if $A$ does not have full rank, then all minors of $A$ of size $r$ vanish, which of course would guarantee that all those minors vanish after specialisation to $(a_1,\ldots,a_p)$, when they become the minors of $A|(a_1,\dots,a_p)$.
\end{proof}

\begin{definition}
	\label{definitiondetideal}
	Let $A$ be an $m \times n$ matrix over $\mathbb{F}[x_1,\dots,x_p]$.
	The \emph{determinantal ideals} $DI_r(A)$ for $r = 0,\dots,\min(m,n)$ are defined as follows:
	$DI_0(A) = \mathbb{F}[x_1,\dots,x_p]$, and if $r \ge 1$ then $DI_r(A)$ is the ideal in $\mathbb{F}[x_1,\dots,x_p]$ generated
	by all $r \times r$ minors of $A$.
\end{definition}

In terms of determinantal ideals, we can reformulate the classical formula for the rank of a matrix as follows.

\begin{proposition}\label{prop:ranklemma}
	Let $A$ be an $m \times n$ matrix over $\mathbb{F}[x_1,\dots,x_p]$. For every value of $r$ not exceeding $\min(m,n)$ we have
	\[
	\mathrm{invrank}_A( r ) = V( DI_{r+1} ) \setminus V( DI_r ).
	\]	
\end{proposition}

The advantage of using determinantal ideals is that they allow us to study the rank of a matrix using only
ring operations (without division).
The classical theory of determinantal ideals is concerned almost exclusively with the homogeneous case,
in which every minor is a homogeneous polynomial; see~\cite{MR}.
Since many entries of the cubic relation matrix $M$ (to be defined in the next section) equal 1,
the determinantal ideals we study in  what follows will be inhomogeneous.
We could reformulate our problem in homogeneous terms by introducing a new parameter $x_0$ to play the role of
the coefficient of $(a_1a_2)a_3$ in Relation \eqref{LRrelation}.
This leads into the theory of sparse determinantal ideals \cite{Boo}.
However, having many leading 1s in the matrix will be very useful from a computational point of view.

From now on, most of our computations require a choice of monomial order.

\begin{definition}
	For an element $m$ of the monomial basis of $\mathbb{F}[x_1,\dots,x_p]$ we write:
	\[
	m = \prod_{k=1}^p x_k^{e_k},
	\qquad
	e(m) = [ e_1, \dots, e_p ],
	\qquad
	\deg(m) = \sum_{k=1}^p e_k.
	\]
	The \emph{graded reverse lexicographic order} (called \texttt{grevlex} in \texttt{Maple}, \texttt{Magma} and \texttt{Macaulay}, \texttt{degrevlex} in \texttt{sage}, and \texttt{dp} in \texttt{Singular}) is defined by:
	$m \prec m'$ if and only if
	\begin{itemize}
		\item  $\deg(m) < \deg(m')$, or
		\item $\deg(m) = \deg(m')$, and $e_k > e'_k$ where $k$ is the smallest index such that $e_k \ne e'_k$.
	\end{itemize}
	Note that $x_1 \prec \cdots \prec x_p$.
	
	The leading monomial $\mathrm{LM}(f)$ of a nonzero polynomial $f \in \mathbb{F}[x_1,\dots,x_p]$ is the
	greatest with respect to $\prec$, and $\mathrm{LC}(f)$ is the coefficient of $\mathrm{LM}(f)$.
\end{definition}

In what follows, we shall use this ordering of monomials for ordering lists of polynomials (term by term).

\begin{definition}
	\label{definitiongrobner}
	Given a monomial order $\prec$, every ideal $J \subseteq \mathbb{F}[x_1,\dots,x_p]$ contains a
	(finite) ordered set $G = \{f_1,\dots,f_t\}$ of (nonzero) polynomials, called a \emph{Gr\"obner basis}
	with respect to $\prec$, satisfying the following conditions:
	\begin{itemize}
		\item
		$J = ( G )$: the polynomials $f_1,\dots,f_t$ generate $J$.
		\item
		$( \, \{ \, \mathrm{LM}(f) \mid f \in J \, \} \, ) =
		( \, \{ \, \mathrm{LM}(f) \mid f \in G \, \} \, )$: the ideal generated by the leading monomials
		of the elements of $J$ is generated by the leading monomials of the elements of $G$.
	\end{itemize}
	A \emph{reduced} Gr\"obner basis satisfies the following additional conditions:
	\begin{itemize}
		\item
		The generators are monic: $\mathrm{LC}(f) = 1$ for every $f \in G$.
		\item
		For every $f \in G$ no monomial of any $f' \in G \setminus \{f\}$ is divisible by $\mathrm{LM}(f)$.
	\end{itemize}
	Every ideal has a unique reduced Gr\"obner basis with respect to a given monomial order.
	Of the many books on Gr\"obner bases, Cox et al.\cite{CLO1,CLO2} are the most approachable.
\end{definition}

\begin{definition}
	For an ideal $J \subseteq \mathbb{F}[x_1,\dots,x_p]$, the \emph{zero set} $V(J)$ is
	the set of points in $\mathbb{F}^p$ which are solutions to every polynomial in $J$:
	\[
	V(J)
	=
	\big\{ \,
	(a_1,\dots,a_p) \in \mathbb{F}^p
	\,\big|\,
	f(a_1,\dots,a_p) = 0 \; \text{for all} \; f \in J
	\, \big\}.
	\]
	The ideal $I(S)$ of the subset $S \subseteq \mathbb{F}^p$ consists of all polynomials which vanish
	on $S$:
	\[
	I(S)
	=
	\big\{ \,
	f \in \mathbb{F}[x_1,\dots,x_p]
	\,\big|\,
	f(a_1,\dots,a_p) = 0 \; \text{for all} \; (a_1,\dots,a_p) \in S
	\, \big\}.
	\]
	Clearly $J \subseteq I(V(J))$.
	The \emph{radical} of $J$ is the ideal $\sqrt{J} = I(V(J))$.
	We say that $J$ is a \emph{radical ideal} if $J = \sqrt{J}$.
	For our purposes, the value of these concepts is that often $\sqrt{J}$ is much larger than
	$J$ and has a much smaller and simpler Gr\"obner basis.
\end{definition}

\begin{algorithm}
	\label{partialsmithform}
	For a matrix whose entries are multivariate polynomials, this algorithm produces a
	\emph{partial Smith form} based on elimination using nonzero scalar entries.
	The basic idea is rather naive, but this algorithm will be useful in reducing the size of matrices
	before computing determinantal ideals.
	
	Input: an $m \times n$ matrix $R$ with entries in $\mathbb{F}[x_1,\dots,x_p]$.
	
	Output: an $m \times n$ matrix $S$ equivalent to $R$ over $\mathbb{F}[x_1,\dots,x_p]$ in the
	sense that $S = URV$ where $U$ ($m \times m$) and $V$ ($n \times n$) are invertible matrices over
	$\mathbb{F}[x_1,\dots,x_p]$, that is, $\det(U)$ and $\det(V)$ are nonzero scalars.
	Furthermore, $S$ consists of two diagonal blocks: an identity matrix and a block $B$ in which no entry is a
	nonzero scalar.
	\begin{enumerate}
		\item
		Set $S \leftarrow R$.
		Set $k \leftarrow 1$.
		\item
		While $s_{ij} \in \mathbb{F} \setminus \{0\}$ for some $i, j \ge k$ do:
		\begin{itemize}
			\item
			Find the least $i \ge k$ for which $s_{ij} \in \mathbb{F} \setminus \{0\}$ for some $j \ge k$.
			\item
			If $i \ne k$ then interchange rows $i$ and $k$ of $S$.
			\item
			Find the least $j \ge k$ for which $s_{kj} \in \mathbb{F} \setminus \{0\}$.
			\item
			If $j \ne k$ then interchange columns $j$ and $k$ of $S$.
			\item
			If $s_{kk} \ne 1$ then divide row $k$ of $S$ by $s_{kk}$.
			\item
			For $i = k+1,\dots,m$ do: subtract $s_{ik}$ times row $k$ from row $i$.
			\item
			For $j = k+1,\dots,n$ do: subtract $s_{kj}$ times column $k$ from column $j$.
			\item
			Set $k \leftarrow k+1$.
		\end{itemize}
		\item
		Return $S$.
	\end{enumerate}
\end{algorithm}

\section{General results on parametrized one-relation operads}
\label{cubicrelationmatrixsection}

\subsection{The cubic relation matrix \texorpdfstring{$M$}{M}}

\begin{notation}
	\label{120basis}
	The monomial basis of the quadratic space $\calT(3)$ consists of the five elements from Example \ref{ex:basismonomials}.
	We replace the argument symbols $\arg$ by the identity permutation of the variables $a_1,a_2,a_3,a_4$ and obtain a
	generating set for the $S_4$-module $\calT(4)$:
\[
    \gamma_1 = ((a_1a_2)a_3)a_4, \quad
	\gamma_2 = (a_1(a_2a_3))a_4, \quad
	\gamma_3 = (a_1a_2)(a_3a_4), \quad 
	\gamma_4 = a_1((a_2a_3)a_4), \quad
	\gamma_5 = a_1(a_2(a_3a_4)).
\]
	To each generator $\gamma_1,\dots,\gamma_5$ we apply all 24 permutations from $S_4$ to obtain a linear basis
	of $\calT(4)$.
	We write these basis monomials using the notation $[\tau]_q = \tau \cdot \gamma_q$
	for $\tau \in S_4$ and $q \in \{1,\dots,5\}$.
	We impose a total order by defining monomial $j \in \{1,\dots,120\}$ to be $[\tau]_q$ where
	$j-1 = 24(q-1) + (r-1)$ and $r \in \{1,\dots,24\}$ and $\tau$ is permutation $r$ in lex order.
\end{notation}

Let us consider the general relation of the type \eqref{LRrelation}
\begin{align*}
\rho(a_1,a_2,a_3)=& \; (a_1a_2)a_3 - x_1 a_1(a_2a_3) - x_2 a_1(a_3a_2)  
- x_3 a_2(a_1a_3) - x_4 a_2(a_3a_1) - x_5 a_3(a_1a_2) - x_6 a_3(a_2a_1).
\end{align*}
In what follows, we denote by $\calJ$ the operad ideal in $\calT$ generated by $\rho$.
We regard the coefficients $x_1,\dots,x_6$ as indeterminates, and so $\calT$ has become an operad not
over $\mathbb{F}$ but instead over the polynomial ring $\mathcal{C} = \mathbb{F}[x_1,\dots,x_6]$.
That is, we replace each $S_n$-module $\calT(n)$ over $\mathbb{F}$ by the
tensor product $\mathcal{C} \otimes\calT(n)$ over $\mathcal{C}$ where
every $\tau \in S_n$ acts as the identity map on $\mathcal{C}$.
Thus $\calT(n)$ has changed from a vector space of dimension $\frac{(2k-2)!}{(k-1)!}$
(Lemma \ref{symmetricdimension}) to a free $\mathcal{C}$-module of the same rank.
In particular, $\calT(4)$ is a free $\mathcal{C}$-module of rank $120$.

According to Definition \ref{definitionconsequences} the elements generate the $S_4$-module $\calJ(4) \subseteq \calT(4)$:
 \[
		\begin{array}{l}
		\rho(a_1a_2,a_3,a_4)
		=
		((a_1a_2)a_3)a_4 - x_1 (a_1a_2)(a_3a_4) - x_2 (a_1a_2)(a_4a_3) 
		\\
		\qquad\qquad
		- x_3 a_3((a_1a_2)a_4) - x_4 a_3(a_4(a_1a_2)) - x_5 a_4((a_1a_2)a_3) - x_6 a_4(a_3(a_1a_2)),
		\\
		\rho(a_1,a_2a_3,a_4)
		=
		(a_1(a_2a_3))a_4 - x_1 a_1((a_2a_3)a_4) - x_2 a_1(a_4(a_2a_3))
		\\
		\qquad\qquad
		 - x_3 (a_2a_3)(a_1a_4) - x_4 (a_2a_3)(a_4a_1) - x_5 a_4(a_1(a_2a_3)) - x_6 a_4((a_2a_3)a_1),
		\\
		\rho(a_1,a_2,a_3a_4)
		=
		(a_1a_2)(a_3a_4) - x_1 a_1(a_2(a_3a_4)) - x_2 a_1((a_3a_4)a_2)
		\\
		\qquad\qquad
		 - x_3 a_2(a_1(a_3a_4)) - x_4 a_2((a_3a_4)a_1) - x_5 (a_3a_4)(a_1a_2) - x_6 (a_3a_4)(a_2a_1),
		\\
		\rho(a_1,a_2,a_3)a_4
		=
		((a_1a_2)a_3)a_4 - x_1 (a_1(a_2a_3))a_4 - x_2 (a_1(a_3a_2))a_4 
		\\
		\qquad\qquad
		- x_3 (a_2(a_1a_3))a_4 - x_4 (a_2(a_3a_1))a_4 - x_5 (a_3(a_1a_2))a_4 - x_6 (a_3(a_2a_1))a_4,
		\\
		a_1\rho(a_2,a_3,a_4)
		=
		a_1((a_2a_3)a_4) - x_1 a_1(a_2(a_3a_4)) - x_2 a_1(a_2(a_4a_3)) 
		\\
		\qquad\qquad
		- x_3 a_1(a_3(a_2a_4)) - x_4 a_1(a_3(a_4a_2)) - x_5 a_1(a_4(a_2a_3)) - x_6 a_1(a_4(a_3a_2)).
		\end{array}
 \]
Using the notation for basis elements described above, these expansions can be written as
\begin{equation}\label{consequences-arity-4-expansion}
\begin{array}{l}
[1234]_1 - x_1 [1234]_3 - x_2 [1243]_3 - x_3 [3124]_4 - x_4 [3412]_5 - x_5 [4123]_4 - x_6 [4312]_5
\\
{}
[1234]_2 - x_1 [1234]_4 - x_2 [1423]_5 - x_3 [2314]_3 - x_4 [2341]_3 - x_5 [4123]_5 - x_6 [4231]_4
\\
{}
[1234]_3 - x_1 [1234]_5 - x_2 [1342]_4 - x_3 [2134]_5 - x_4 [2341]_4 - x_5 [3412]_3 - x_6 [3421]_3
\\
{}
[1234]_1 - x_1 [1234]_2 - x_2 [1324]_2 - x_3 [2134]_2 - x_4 [2314]_2 - x_5 [3124]_2 - x_6 [3214]_2
\\
{}
[1234]_4 - x_1 [1234]_5 - x_2 [1243]_5 - x_3 [1324]_5 - x_4 [1342]_5 - x_5 [1423]_5 - x_6 [1432]_5
\end{array}
\end{equation}
The following list of $120$ relations generates $\calJ(4)$ as a $\mathcal{C}$-module:
\[
\rho(a_1a_2,a_3,a_4) . \tau,  \quad 
\rho(a_1,a_2a_3,a_4) . \tau, \quad 
\rho(a_1,a_2,a_3a_4) . \tau, \quad 
\rho(a_1,a_2,a_3)a_4 . \tau, \quad 
a_1\rho(a_2,a_3,a_4) . \tau ,
\]
here $\tau \in S_4$ is an arbitrary permutation.
These relations can be represented as row vectors of dimension 120 over $\mathcal{C}$ using the total
order of Notation \ref{120basis}; each vector has the entries $\{ 1, -x_1, \dots, -x_6 \}$ and 113 zeros.
We sort these row vectors into semi-triangular form using the following \emph{total order} $x \prec y$ on
row vectors of the same but arbitrary length:
\begin{itemize}
	\item
	Let $i, j \ge 1$ be the least integers for which $x_i \ne 0$ and $y_j \ne 0$.
	\item
	If $i \ne j$ then $x \prec y$ if and only if $i < j$.
	\item
	If $i = j$ but $x_i \ne y_j$ then $x \prec y$ if and only if $x_i \prec y_j$ according to
	\[
	1 \;\prec\; -x_1 \;\prec\;  -x_2 \;\prec\; -x_3 \;\prec\; -x_4 \;\prec\; -x_5 \;\prec\; -x_6.
	\]
	\item
	If $i = j$ and $x_i = y_j$ then $x \prec y$ if and only if $x' \prec y'$ where $x'$ (resp.~$y'$) is
	obtained from $x$ (resp.~$y$) by deleting the first $i$ entries.
\end{itemize}

\begin{definition} \label{definitionrelationmatrix}
	The \emph{cubic relation matrix} $M = ( m_{ij} )$ is the square matrix of size $120$ in which entry $m_{ij}$
	is the coefficient of the $j$-th basis monomial (Notation \ref{120basis}) in the $i$-th row vector in the
	list of consequences of the relation $\rho$ sorted as above.
\end{definition}

\begin{lemma}
	\label{basicMinfo}
	The cubic relation matrix $M$ has minimal rank $84$.
	The reduced Gr\"obner basis for the first nontrivial determinantal ideal $DI_{85}(M)$ is as follows:
	\[
	x_2 + x_3, \quad x_1 + x_4, \quad x_6, \quad
	x_1^2, \quad x_2 x_1, \quad x_1 x_5 + x_2, \quad x_2^2, \quad x_5 x_2 + x_1, \quad x_5^2 - 1.
	\]
	The reduced Gr\"obner basis for the radical $\sqrt{DI_{85}(M)}$ is as follows:
	\[
	x_1, \quad x_2, \quad x_3, \quad x_4, \quad x_6, \quad x_5^2-1.
	\]
	The only parameter values which produce the minimal rank are $[0,0,0,0,\pm 1,0]$,
	and for these values we obtain the maximal dimension $\mathrm{nullity}(M) = 36$ for the $S_4$-module
	$\calT(4)/\calJ(4)$.
\end{lemma}

\begin{proof}
The following properties of $M$ were obtained using the computer algebra system \texttt{Maple}.
Applying Algorithm \ref{partialsmithform} to the cubic relation matrix $M$ produces $\mathrm{diag}(I_{84},B)$
where $B$ is a $36 \times 36$ matrix in which no entry is a nonzero scalar.
From this is follows that for all $a_1,\dots,a_6 \in \mathbb{F}$ the matrix $M|(a_1,\dots,a_6)$ has rank
$\ge 84$, and $DI_k(M)=DI_{k-84}(B)$ for $k>84$. This formula easily allows to compute the Gr\"obner bases 
for the ideal $DI_{85}(M)=DI_1(B)$ and for its radical, which are as stated in the Lemma. Examining these Gr\"obner bases, 
we see that $B=0$ only for the values of parameters $[0,0,0,0,\pm 1,0]$, which completes the proof. 
\end{proof}

\begin{remark}
	\label{blockremark}
The computations for the proof of Lemma \ref{basicMinfo} also provided the following information about the
block $B$.
Every entry $f$ of $B$ has integer coefficients, and only 99 of the 1296 entries are zero. After normalizing these polynomials
by making all leading coefficients equal to $1$, there are 709 distinct polynomials with 665 distinct irreducible factors;
in fact, 492 of these polynomials are irreducible. In a way, it is remarkable that the ideal generated by these polynomials has such a small and simple
Gr\"obner basis.
\end{remark}

\begin{remark}\label{hugenumbers}
The other determinantal ideals are much harder.
A generating set for $DI_r(B)$ contains $C(36,r)^2$ determinants of $r \times r$ submatrices.
In particular, regularity requires $\mathrm{nullity}(M) = 24$ and hence $\mathrm{rank}(B) = 12$.
To determine the parameter values satisfying this condition, Proposition \ref{prop:ranklemma} tells us to find the
zero sets $V(DI_r(B))$ for $r = 12, 13$.
For $r = 12$ (and worse for $r = 13$) we must evaluate more than $10^{18}$ minors, and $12 \times 12$
determinants over $\mathbb{F}[x_1,\dots,x_6]$ are not easy to compute.
Even supposing that this was possible, we would still have to compute Gr\"obner bases for the two ideals,
and hope that these would make it possible to solve explicitly for the zero sets.
We will be able to overcome these obstacles using the representation theory of the symmetric group,
starting in Section \ref{representationsection}.
\end{remark}

\subsection{Nilpotency theorem}

\begin{theorem}\label{th:niltheorem}
	Let $\mathcal{N}$ be the set of all points $\mathbf{a}$ in the parameter space $\,\mathbb{F}^6$ for which
	the operad $\calO_\mathbf{a}$ is nilpotent of index $3$.
	Then $\mathcal{N} $ is a Zariski open subset of the parameter space $\mathbb{F}^6$.
\end{theorem}

\begin{proof}
	Example \ref{antiassociativeexample} showed that the anti-associative identity is a special case of Relation \eqref{LRrelation} and that the anti-associative operad is nilpotent of index $3$.
	Hence setting $x_1 = -1$ and $x_2 = \dots = x_6 = 0$ in $M$ produces an invertible
	matrix over $\mathbb{F}$.
	It follows from Proposition \ref{fullranklemma} that the cubic relation matrix $M$ is invertible over the field of rational functions $\mathbb{F}(x_1,\dots,x_6)$. For $\mathbf{a}=(a_1,\ldots,a_6)\in\mathbb{F}^6$, the parametrized one-relation operad $\calO_\mathbf{a}$ is
	nilpotent of index $3$ if and only if $\det(M|(a_1,\dots,a_6)) \ne 0$; this condition defines a Zariski open subset in the space of parameters.
\end{proof}

\subsection{Towards classifying regular parametrized one-relation operads}

\begin{lemma}
	If the operad $\calO_\mathbf{x}$ is regular for the values $x_k = a_k \in \mathbb{F}$ $(1 \le k \le 6)$ then
	\begin{itemize}
		\item
		$\mathrm{rank}(M) = 96$ and $\mathrm{rowspace}(M) \cong \big(\mathbb{F}S_4\big)^4$ as an $S_4$-module.
		\item
		$\mathrm{nullity}(M) = 24$ and $\mathrm{nullspace}(M) \cong \mathbb{F}S_4$ as an $S_4$-module.
	\end{itemize}
\end{lemma}

\begin{proof}
	Regularity means that $\calT(n)\, / \calJ(n) \cong \mathbb{F} S_{n}$ for all $n \ge 1$,
	and so in particular we have $\calT(4)\, / \calJ(4) \cong \mathbb{F} S_4$.
	Since $\Omega(4)$ is five-dimensional, we have $\calT(4) \cong \big( \mathbb{F} S_4 \big)^5$, which implies
	$\calJ(4) \cong \big( \mathbb{F} S_4\big)^4$.
	Since $\mathrm{rowspace}(M) \cong \calJ(4)$, we have
	$\mathrm{nullspace}(M) \cong \calT(4)\, / \calJ(4)$.
\end{proof}

Consider any subset $A \subseteq \{1,\dots,6\}$ and let $M(A)$ be the matrix obtained by setting
$x_i = 0$ for all $i \in A$ in the cubic relation matrix $M$.
If we apply Algorithm \ref{partialsmithform} to $M(A)$ then we obtain a block diagonal matrix
$\mathrm{diag}( I_r, B_s )$ where $r = r(A)$, $s = s(A)$ and $r+s = 120$.
As before, $I_r$ is the identity matrix of size $r$, and $B_s$ is a square matrix of size $s$ in which no
entry is a nonzero scalar.
The following result can be obtained by
a straightforward \texttt{Maple} computation.

\begin{lemma}
	\label{sizeBlemma}
	The size $s$ of $B$ depends only on whether 5 or 6 are in $A$:
	\begin{itemize}
		\item
		If $x_5 = x_6 = 0$ then $B$ has size $24$.
		\item
		If $x_5 = 0$ but $x_6 \ne 0$ then $B$ has size $30$.
		\item
		If $x_5 \ne 0$ then $B$ has size $36$.
	\end{itemize}
\end{lemma}

We now consider the $16$ cases in which $x_5 = x_6 = 0$; we can deal with them all at once by allowing
$x_1, \dots, x_4$ to be free parameters. We shall be able to establish the following rather attractive result,
which shows how the four most familiar cases of parametrized one-relation operads may be obtained directly from elementary
observations using linear and commutative algebra. In fact, we shall provide two proofs of this result, since each of them
is somewhat instructive.

\begin{proposition}
	\label{x_5x_6zero}
	The only cases of Relation \eqref{LRrelation} with $x_5 = x_6 = 0$ which are regular  
	are those defining the trivial, associative, Leibniz and Zinbiel operads.
\end{proposition}

\begin{proof}[First proof of Proposition \ref{x_5x_6zero}]
\textit{} Algorithm \ref{partialsmithform} reduces $M$ to an identity matrix of size $96$ and a lower right block
$B$ of size $24$. Thus in order for the nullity of $M$ to be $24$ it is necessary and sufficient that $B = 0$, and this in turn is equivalent to $DI_1(B) = \{0\}$.
In $B$, $432/576=3/4$ of the entries are nonzero but there are only $18$ distinct nonzero entries,  with degrees $\{3,4\}$, coefficients $\{\pm 1,2\}$, and numbers of terms $\{2,6,7,8\}$. Figure \ref{X1234basis} lists these entries in \texttt{grevlex} order; those not factored are irreducible.

\begin{figure}[ht]
	$
		\begin{array}{rl}
		1&
		x_4 x_1 \big( x_2 + x_3 \big)
		\\[3pt]
		2&
		x_4^2 x_2 - x_3^2 x_2 - x_3^2 x_1 + x_4 x_2 x_1 - x_4 x_1^2 - x_3 x_1^2
		\\[3pt]
		3&
		\big( x_2 + x_3 \big)  \big( x_1^2 - x_1 x_2 + x_1 x_3 - x_2 x_4 + x_3 x_4 \big)
		\\[3pt]
		4&
		x_4^2 x_3 - x_3 x_2^2 + x_4 x_3 x_1 + x_2^2 x_1 + x_4 x_1^2 - x_2 x_1^2
		\\[3pt]
		5&
		x_4 x_3 x_2^2 + x_4^2 x_3 x_1 + x_3 x_2^2 x_1 + x_4 x_3 x_1^2 + x_4 x_3^2 + x_4^2 x_1 - x_3 x_2
		\\[3pt]
		6&
		x_4^2 x_2^2 + x_3 x_2^3 + x_4^3 x_1 + x_4 x_3 x_2 x_1 + x_3^3 + x_4 x_3 x_1 - x_3 x_1
		\\[3pt]
		7&
		x_4 x_3^2 x_2 + x_4^2 x_2 x_1 + x_3^2 x_2 x_1 + x_4 x_2 x_1^2 - x_4 x_2^2 - x_4^2 x_1 + x_3 x_2
		\\[3pt]
		8&
		x_4^2 x_3 x_2 + 2 x_4 x_3 x_2 x_1 + x_3 x_2 x_1^2 + x_4^2 x_1 + x_3 x_2 x_1 - x_3 x_2
		\\[3pt]
		9&
		x_4 \big( x_1 x_2^2 + x_1 x_2 x_3 + x_2^2 x_4 + x_2 x_3 x_4 - x_1 x_2 - x_1 x_3 + x_2 \big)
		\\[3pt]
		10&
		x_4^3 x_2 + x_4 x_3 x_2^2 + x_4^2 x_2 x_1 + x_3 x_2^2 x_1 + x_4 x_3 x_1 + x_3 x_1^2 - x_3 x_1
		\\[3pt]
		11&
		x_4 \big( x_2 + x_3 \big)  \big( 2 x_1 x_4 + x_2^2 + x_3^2 \big)
		\\[3pt]
		12&
		x_4^2 x_3^2 + x_3^3 x_2 + x_4^3 x_1 + x_4 x_3 x_2 x_1 - x_2^3 - x_4 x_2 x_1 + x_2 x_1
		\\[3pt]
		13&
		x_4 \big( x_1 x_2 x_3 + x_1 x_3^2 + x_2 x_3 x_4 + x_3^2 x_4 + x_1 x_2 + x_1 x_3 - x_3 \big)
		\\[3pt]
		14&
		x_4^2 x_3^2 + 2 x_4^2 x_3 x_2 + x_4^2 x_2^2 + x_3^2 x_2^2 + x_4 x_3^2 x_1 + x_4 x_2^2 x_1 + x_4^2 x_1^2 - x_4
		\\[3pt]
		15&
		x_4^3 x_3 + x_4 x_3^2 x_2 + x_4^2 x_3 x_1 + x_3^2 x_2 x_1 + x_4 x_2 x_1 + x_2 x_1^2 - x_2 x_1
		\\[3pt]
		16&
		x_4^3 x_3 + x_4^3 x_2 + x_4 x_3^2 x_2 + x_4 x_3 x_2^2 - x_3 x_2 x_1 - x_2^2 x_1 + x_2^2
		\\[3pt]
		17&
		x_4^3 x_3 + x_4^3 x_2 + x_4 x_3^2 x_2 + x_4 x_3 x_2^2 + x_3^2 x_1 + x_3 x_2 x_1 - x_3^2
		\\[3pt]
		18&
		x_4^4 + 2 x_4^2 x_3 x_2 + x_3^2 x_2^2 + x_3 x_2 x_1 + x_1^3 - x_1^2
		\end{array}
	$
	\caption{Nonzero entries of lower right block $B$ when $x_5 = x_6 = 0$}
	\label{X1234basis}
\end{figure}

The Gr\"obner basis for the ideal generated by these entries has seven elements:
\[
x_4, \;\;
x_2 \big( x_2{-}x_1 \big), \quad
x_3 x_2, \quad
x_3 \big( x_3{+}x_1 \big), \quad
x_1^2 \big( x_1{-}1 \big), \quad
x_2 x_1 \big( x_1{-}1 \big), \quad
x_3 x_1 \big( x_1{-}1 \big).
\]
The Gr\"obner basis for the radical also has seven elements:
\[
x_4, \quad
x_1 \big( x_1{-}1 \big), \quad
x_2 \big( x_1{-}1 \big), \quad
x_3 \big( x_1{-}1 \big), \quad
x_2 \big( x_2{-}1 \big), \quad
x_3 x_2, \quad
x_3 \big( x_3{+}1 \big)
\]
From these results it is easy to verify that $DI_1(B)$ is zero-dimensional and that its zero set $V(DI_1(B))$
consists of exactly four points $( x_1, x_2, x_3, x_4 ) = (0,0,0,0)$, $(1,0,0,0)$, $(1,1,0,0)$, $(1,0,-1,0)$.
We have seen these coefficients before, in Example \ref{fiveexamples}: they correspond to the left nilpotent,
associative, Zinbiel and Leibniz operads. 
\end{proof}

\begin{proof}[Second proof of Proposition \ref{x_5x_6zero}] 
While Relation \eqref{LRrelation}  allows reassociation
of parentheses to the right when we deal with products of three arguments, that does not in general help to reassociate parentheses in products of more than three arguments: since we allow all permutations of arguments on the right side, an infinite chain of reassociations might happen. However, if we assume that $x_5=x_6=0$, that cannot happen, as the following lemma shows.

\begin{lemma}
	Suppose that $x_5=x_6=0$. Then every operation in the corresponding operad is equal to a linear combination of right-normed products.
\end{lemma}

\begin{proof}
	Let us consider some balanced bracketing of $k\ge 2$ arguments. It is of the form $(AB)$, where $A$ and $B$ are balanced bracketings of fewer arguments, with $l$ arguments in $A$ and $k-l$ arguments in $B$. We shall prove the statement by induction on $k$, and for a fixed $k$, by induction on $l$. In both cases, the basis of induction is trivial: for $k=2$ there is nothing to prove, and for each $k$ and $l=1$ we may use the induction hypothesis and write $B$ as a linear combination of right-normed products; the right-normed property does not change when we multiply by~$A$.
	
	Assume that $l\ge 2$, so  that $A=(A_1A_2)$; we are in a situation where we can apply the defining relation of our operad, obtaining
	\[
	(A_1A_2)B=
	x_1 A_1(A_2B) +
	x_2 A_1(BA_2) +
	x_3 A_2(A_1B) +
	x_4 A_2(BA_1).
	\]
The first four permutations are exactly those which do not bring the third argument into the first position, so each of these terms has the parameter $l$ smaller than the original one, and the induction hypothesis applies.
\end{proof}

This lemma shows that under the assumption $x_5=x_6=0$ the spanning property of the right-normed products is trivially satisfied, so there is a surjective map from the regular representation of $S_n$ onto the $n$-th component of our operad. It remains to check that this map has no kernel. Let us start with arity $4$.  Note that the defining relation of our operad can be applied as a rewriting rule to the product $((a_1a_2)a_3)a_4$ in two different ways, by rewriting $(a_1a_2)a_3$ first, or by rewriting $(ba_3)a_4$ and setting $b=(a_1a_2)$, as in Example \ref{antiassociativeexample}. This leads to two \emph{a priori} different expressions for $((a_1a_2)a_3)a_4$ as linear combinations of right-normed products; we collect the nonzero coefficients of the difference of those in the the following table, where the polynomial in the row indexed $\tau\in S_4$ corresponds to the coefficient of $a_1(a_2(a_3a_4)) . \tau$ :
\[
\begin{array}{ll}
1234&x_2^2x_3^2 + 2x_2x_3x_4^2 + x_4^4 + x_1^3 + x_1x_2x_3  - x_1^2\\
1243&x_1x_2x_3^2 + x_1x_3x_4^2 + x_2x_3^2x_4 + x_3x_4^3 + x_1^2x_2 + x_1x_2x_4 - x_1x_2\\
1324&x_1^2x_2 + x_1^2x_3 - x_1x_2^2 + x_1x_3^2 - x_2^2x_4 + x_3^2x_4\\
1342&- x_1^2x_2 + x_1^2x_4 + x_1x_2^2 + x_1x_3x_4 - x_2^2x_3 + x_3x_4^2\\
1423&x_1x_2x_3x_4 + x_1x_4^3 + x_2x_3^3 + x_3^2x_4^2 - x_1x_2x_4 - x_2^3 + x_1x_2\\
1432&x_2^2x_3x_4 + x_2x_3^2x_4 + x_2x_4^3 + x_3x_4^3 - x_1x_2^2 - x_1x_2x_3 + x_2^2\\
2134&x_1x_2^2x_3 + x_1x_2x_4^2 + x_2^2x_3x_4 + x_2x_4^3 + x_1^2x_3 + x_1x_3x_4 - x_1x_3\\
2143&x_1^2x_2x_3 + 2x_1x_2x_3x_4 + x_2x_3x_4^2 + x_1x_2x_3 + x_1x_4^2 - x_2x_3\\
2314&x_1^2x_3 + x_1^2x_4 - x_1x_2x_4 + x_1x_3^2 + x_2x_3^2  - x_2x_4^2\\
2341&x_1x_2x_4 + x_1x_3x_4\\
2413&x_1^2x_2x_4 + x_1x_2x_3^2 + x_1x_2x_4^2 + x_2x_3^2x_4 - x_1x_4^2 - x_2^2x_4 + x_2x_3\\
2431&x_1x_2^2x_4 + x_1x_2x_3x_4 + x_2^2x_4^2 + x_2x_3x_4^2 - x_1x_2x_4 - x_1x_3x_4 + x_2x_4\\
3214&x_2^2x_3x_4 + x_2x_3^2x_4 + x_2x_4^3 + x_3x_4^3 + x_1x_2x_3 + x_1x_3^2 - x_3^2\\
3241&x_1x_2x_3x_4 + x_1x_3^2x_4 + x_2x_3x_4^2 + x_3^2x_4^2 + x_1x_2x_4 + x_1x_3x_4 - x_3x_4\\
3124&x_1x_2x_3x_4 + x_1x_4^3 + x_2^3x_3 + x_2^2x_4^2 + x_1x_3x_4 + x_3^3 - x_1x_3\\
3142&x_1^2x_3x_4 + x_1x_2^2x_3 + x_1x_3x_4^2 + x_2^2x_3x_4 + x_1x_4^2 + x_3^2x_4 - x_2x_3\\
3412&2x_1x_2x_4^2 + 2x_1x_3x_4^2 + x_2^2x_3x_4 + x_2x_3^2x_4 + x_2^3x_4 + x_3^3x_4\\
3421&x_1^2x_4^2 + x_1x_2^2x_4 + x_1x_3^2x_4 + 2x_2x_3x_4^2 + x_2^2x_3^2 + x_2^2x_4^2 + x_3^2x_4^2 - x_1x_4\\
\end{array}
\]

If the right-normed products are linearly independent, all those coefficients must be equal to zero.  The Gr\"obner basis for the corresponding system of polynomial equations is
\[
x_1^4 - x_1^3, \quad x_4^4 + x_1^3 - x_1^2, \quad  x_2^3 - x_2^2, \quad x_3^3 + x_3^2, \quad  x_1x_2 - x_2^2,\quad x_1x_3 + x_3^2,\quad  x_2x_3, \quad x_1x_4, \quad x_2x_4,\quad x_3x_4 .
\]
This implies that every solution to this system has 
 \[
x_4=0,  \quad x_1\in\{0,1\}, \quad  x_2\in\{0,1\}, \quad  x_3\in\{-1,0\},
 \] and $x_2x_3=x_3(x_1+x_3)=x_2(x_1-x_2)=0$, so the only solutions are $(0,0,0,0)$, which is the left-nilpotent operad, $(1,0,0,0)$, which is the associative operad, $(1,1,0,0)$, which is the Zinbiel operad, and $(1,0,-1,0)$, which is the Leibniz operad. 
\end{proof}

Proposition \ref{x_5x_6zero} shows that if we wish to find new regular solutions we will have
to consider the more difficult cases in which either $x_5$ or $x_6$ is nonzero. Examining the two proofs
of that proposition, we see that since, according to Lemma \ref{sizeBlemma},  the matrix $B$ has size either 30 or 36 in these cases 
we have to deal with either impractically large numbers of minors (more than $10^{18}$ in
the worst case of matrix of size 36),
or a rewriting rule that has no termination property.
We therefore need to introduce some more powerful techniques, and that is the topic of the next section.

\section{Representation theory of the symmetric groups}
\label{representationsection}

Because of the symmetric group actions on the components of any operad, it is to be expected that
representation theory of symmetric groups can be utilized in operad theory.
For an operad presented as a quotient of a free operad, the $n$-th component of the ideal
of relations is an $S_n$-submodule of the direct sum of a finite number of copies of the regular
$S_n$-module, $\mathbb{F} S_n$. In simplest terms, the motivation for using representation theory is to
``divide and conquer'':
to split one large intractable problem into a number of smaller tractable pieces which are collectively
equivalent to the original problem.
We refer the reader to \cite{BMP} for a systematic development of the necessary material using modern
notation and terminology.

There are two significant advantages to using the representation theory of the symmetric group to study
algebraic operads.
We have already mentioned the first: this method allows us to study a set of multilinear relations
``one representation at a time'', which greatly reduces the sizes of the matrices involved.
The second important reason is that using representation theory allows us to specify beforehand the
$S_n$-module structure
of the space of relations, not only its dimension, and this can save a great deal of further computation.

For example, the regular $S_4$-module $\mathbb{F} S_4$ has dimension 24, but there are other
$S_4$-modules of dimension 24.
Indeed, if $m_1,\dots,m_5 \ge 0$ are the multiplicities of the simple modules
$[4],[31],[2^2],[21^2],[1^4]$
in the $S_4$-module $T$,
\[
T \cong  m_1[4] \oplus m_2[31] \oplus m_3[2^2] \oplus m_4[21^2] \oplus m_5[1^4],
\]
then
\[
\dim(T) = 24 \quad \iff \quad m_1 + 3 m_2 + 2 m_3 + 3 m_4 + m_5 = 24.
\]
There are $1615$ solutions to this equation, and no two of the corresponding modules are isomorphic,
but only $\mathbb{F}S_4$ has multiplicities $[1,3,2,3,1]$.
If we consider only submodules of $( \mathbb{F} S_4 )^5$ then we still have $529$ solutions.
If we restrict further to modules $T$ which are symmetric in the sense that $T \otimes [1^4] \cong T$ where
$[1^4]$ is the sign module, or equivalently $m_1 = m_5$, $m_2 = m_4$, then the number of solutions decreases
to a more manageable 21.

Without representation theory, if we encounter a module of dimension 24, we have to determine its structure by
computing the traces of the representation matrices for a set of conjugacy class representatives and then using
the character table of $S_4$ to express the character as a linear combination (with non-negative
integer coefficients) of the simple characters.
With representation theory, this extra work is unnecessary.

\subsection{Structure theory}

When the characteristic of $\mathbb{F}$ is 0 or $p > n$, the group algebra $\mathbb{F} S_n$ is semisimple,
and the classical structure theory applies.
Let $\lambda$ range over the partitions of $n$; we write $p(n)$ for the number of partitions.
The regular module $\mathbb{F} S_n$ decomposes into the (orthogonal) direct sum of simple two-sided
ideals $M(\lambda)$, each of which is isomorphic to a full matrix algebra $M_{d(\lambda)}(\mathbb{F})$,
where $d_\lambda$ is the dimension of the simple $S_n$-module $[\lambda]$:
\begin{equation}
\label{sndecomp}
\mathbb{F} S_n \cong \bigoplus_{\lambda} M(\lambda),
\qquad \qquad
M(\lambda) \cong M_{d(\lambda)}(\mathbb{F}).
\end{equation}
As a right (or left) ideal $M(\lambda)$ decomposes as the direct sum of $d(\lambda)$ copies of $[\lambda]$
which correspond to the rows (or columns) of $M_{d(\lambda)}(\mathbb{F})$.
Efficient algorithms are known for computing the isomorphism \eqref{sndecomp} in both directions;
see \cite{BMP}.
We will only require the projections which take a partition $\lambda$ and a permutation $\sigma$ and
produce the matrix $R_\lambda(\sigma)$ in $M_{d(\lambda)}(\mathbb{F})$ which represents
the action of $\sigma$ on $[\lambda]$.
The simplest algorithm for computing the matrices $R_\lambda(\sigma)$ was discovered by Clifton \cite{C}.

The isomorphism \eqref{sndecomp} expresses $\mathbb{F} S_n$, a single vector space of dimension $n!$,
as the direct sum of $p(n)$ subspaces of dimensions $d(\lambda)^2$, and these subspaces are orthogonal
in the sense that $xy = 0$ if $x \in M(\lambda)$ and $y \in M(\lambda')$ with $\lambda \ne \lambda'$.
Thus we have divided the original structure of size $n!$ into a list of $p(n)$ independent structures of
average size $n!/p(n)$.
But we have also converted the vector space $\mathbb{F} S_n$ (a tensor of rank 1) into a list of $p(n)$
full matrix algebras (tensors of rank 2).
Thus the original problem has decomposed into $p(n)$ problems of size $\sqrt{n!/p(n)}$, which is
the average dimension
of a simple $S_n$-module.

\subsection{Representation matrices for polynomial identities of arity 4}

We now restrict to the case $n = 4$ that we need to continue our analysis of the cubic relation matrix $M$.
For each partition $\lambda$ of 4 the dimension $d_\lambda$ of the simple module $[\lambda]$ is the
number of standard tableaux; see Figure \ref{tableaux}.

\begin{figure}[ht]
	\[
	\begin{array}{|c|c|c|c|c|}
	\midrule
	\vcenter{\hbox{  \young(1234)}}
	&
	\vcenter{\hbox{\young(123,4) \; \young(124,3) \; \young(134,2)}}
	&
	\vcenter{\hbox{\young(12,34) \; \young(13,24) }}
	&
	\vcenter{\hbox{\young(12,3,4) \; \young(13,2,4) \; \young(14,2,3)}}
	&
	\vcenter{\hbox{\young(1,2,3,4)}}
	\\
	d_\lambda = 1 & d_\lambda = 3 & d_\lambda = 2 & d_\lambda = 3 & d_\lambda = 1
	\\
	\lambda = 4 & \lambda = 31 & \lambda = 2^2 & \lambda = 21^2 & \lambda = 1^4
	\\
	\midrule
	\end{array}
	\]
	\caption{Partitions, dimensions, standard tableaux ($n = 4$)}
	\label{tableaux}
\end{figure}

The corresponding isomorphism \eqref{sndecomp} has the form
\begin{equation}
\label{s4decomp}
\mathbb{F} S_4 \;\cong\;
\mathbb{F} \;\oplus\;
M_3(\mathbb{F}) \;\oplus\;
M_2(\mathbb{F}) \;\oplus\;
M_3(\mathbb{F}) \;\oplus\;
\mathbb{F} ,
\end{equation}
which can be viewed as a map from permutations $\sigma$ to quintuples of matrices $R_\lambda(\sigma)$.
The representation matrices for the generators $\sigma = (12), (23), (34) \in S_4$ are as follows:
\begin{align*}
(12) &\longmapsto
\left[ \quad
\left[ \begin{array}{r} 1 \end{array} \right], \quad
 \begin{bmatrix} 1 & 0 & -1 \\ 0 & 1 & -1 \\ 0 & 0 & -1  \end{bmatrix}, \quad
 \begin{bmatrix}1 & -1 \\ 0 & -1  \end{bmatrix}, \quad
 \begin{bmatrix} 1 & -1 & 1 \\ 0 & -1 & 0 \\ 0 & 0 & -1  \end{bmatrix}, \quad
\left[ \begin{array}{r} -1 \end{array} \right]
\quad \right]
\\
(23) &\longmapsto
\left[ \quad
\left[ \begin{array}{r} 1 \end{array} \right], \quad
\begin{bmatrix} 1 & 0 & 0 \\ 0 & 0 & 1 \\ 0 & 1 & 0 \end{bmatrix}, \quad
\begin{bmatrix}0 & 1 \\ 1 & 0 \end{bmatrix}, \quad
\begin{bmatrix} 0 & 1 & 0 \\ 1 & 0 & 0 \\ 0 & 0 & -1 \end{bmatrix}, \quad
\left[ \begin{array}{r} -1 \end{array} \right]
\quad \right]
\\
(34) &\longmapsto
\left[ \quad
\left[ \begin{array}{r} 1 \end{array} \right], \quad
\begin{bmatrix} 0 & 1 & 0 \\ 1 & 0 & 0 \\ 0 & 0 & 1 \end{bmatrix}, \quad
\begin{bmatrix}1 & -1 \\ 0 & -1 \end{bmatrix}, \quad
\begin{bmatrix}-1 & 0 & 0 \\ 0 & 0 & 1 \\ 0 & 1 & 0 \end{bmatrix}, \quad
\left[ \begin{array}{r} -1 \end{array} \right]
\quad \right]
\end{align*}

Recall from Notation \ref{120basis} that the $S_4$-module $\calT(4)$ is isomorphic to the direct sum
of five copies of $\mathbb{F} S_4$ generated by the five basis monomials $\gamma_1,\dots,\gamma_5$ of $\Omega(4)$.
Thus every multilinear polynomial identity $I$ of arity 4 can be decomposed into a sum of five components,
$I = I_1 + \cdots + I_5$, where each $I_i$ can be identified with an element of $\mathbb{F}S_4$ and each
monomial in $I_i$ has the same bracketing as $\gamma_i$.
We combine this decomposition of $\calT(4)$ with the decomposition \eqref{s4decomp} and rearrange the
components to obtain the isotypic decomposition of $\calT(4)$:
\[
\calT(4)
\cong
\bigoplus_{j=1}^5
\Big(
\mathbb{F} \oplus M_3(\mathbb{F}) \oplus M_2(\mathbb{F}) \oplus M_3(\mathbb{F}) \oplus \mathbb{F}
\Big)
\cong
\mathbb{F}^5 \oplus
M_3(\mathbb{F})^5 \oplus
M_2(\mathbb{F})^5 \oplus
M_3(\mathbb{F})^5 \oplus
\mathbb{F}^5 .
\]
To obtain the analogous decomposition of the multilinear identity $I = I_1 + \cdots + I_5$, we compute the
representation matrices $R_\lambda(I_j)$ for $\lambda = 4,\dots,1^4$ and $j=1,\dots,5$.
The isotypic decomposition of $I$ is a sequence of five matrices indexed by $\lambda$ of sizes
$d_\lambda \times 5 d_\lambda$:
\[
I
\longmapsto
\left\{ \;\;
\begin{array}{ll}
\Big[ \,
R_4(I_1) \,\big|\, R_4(I_2) \,\big|\, R_4(I_3) \,\big|\, R_4(I_4) \,\big|\, R_4(I_5)
\, \Big]
&\qquad 1 \times 5
\\[6pt]
\Big[ \,
R_{31}(I_1) \,\big|\, R_{31}(I_2) \,\big|\, R_{31}(I_3) \,\big|\, R_{31}(I_4) \,\big|\, R_{31}(I_5)
\, \Big]
&\qquad 3 \times 15
\\[6pt]
\Big[ \,
R_{2^2}(I_1) \,\big|\, R_{2^2}(I_2) \,\big|\, R_{2^2}(I_3) \,\big|\, R_{2^2}(I_4) \,\big|\, R_{2^2}(I_5)
\, \Big]
&\qquad 2 \times 10
\\[6pt]
\Big[ \,
R_{21^2}(I_1) \,\big|\, R_{21^2}(I_2) \,\big|\, R_{21^2}(I_3) \,\big|\, R_{21^2}(I_4) \,\big|\, R_{21^2}(I_5)
\, \Big]
&\qquad 3 \times 15
\\[6pt]
\Big[ \,
R_{1^4}(I_1) \,\big|\, R_{1^4}(I_2) \,\big|\, R_{1^4}(I_3) \,\big|\, R_{1^4}(I_4) \,\big|\, R_{1^4}(I_5)
\, \Big]
&\qquad 1 \times 5
\end{array}
\right.
\]
If $\calG = \{ \, I^{(1)}, \dots, I^{(r)} \, \}$ is a set of multilinear identities of arity 4 then
for each $\lambda$ and each $i = 1,\dots,r$ we compute the $d_\lambda \times 5 d_\lambda$ matrix as above
and stack them together to obtain a matrix of size $r d(\lambda) \times 5 d(\lambda)$:
\begin{equation}
\label{Gmatrix}
R_\lambda(\calG)
=
\left[
\begin{array}{c|c|c|c|c}
R_\lambda(I^{(1)}_1) & R_\lambda(I^{(1)}_2) & R_\lambda(I^{(1)}_3) & R_\lambda(I^{(1)}_4) & R_\lambda(I^{(1)}_5)
\\
\vdots & \vdots & \vdots & \vdots & \vdots
\\
R_\lambda(I^{(i)}_1) & R_\lambda(I^{(i)}_2) & R_\lambda(I^{(i)}_3) & R_\lambda(I^{(i)}_4) & R_\lambda(I^{(i)}_5)
\\
\vdots & \vdots & \vdots & \vdots & \vdots
\\
R_\lambda(I^{(r)}_1) & R_\lambda(I^{(r)}_2) & R_\lambda(I^{(r)}_3) & R_\lambda(I^{(r)}_4) & R_\lambda(I^{(r)}_5)
\end{array} \right]
\end{equation}
The row space of this matrix is the isotypic component for partition $\lambda$ of the submodule of $\calT(4)$
generated by $\calG$, and the rank of this matrix is the multiplicity of the simple $S_4$-module $[\lambda]$
in that isotypic component.

\subsection{Regularity in terms of representation theory}

Recall that Relation \eqref{LRrelation} has five consequences \eqref{generalconsequences} in arity 4 which generate the
$S_4$-module $\calJ(4) \subseteq \calT(4)$ of relations in arity 4 for parametrized one-relation
algebras. We rewrite the expansions \eqref{consequences-arity-4-expansion} of those five consequences by collecting terms corresponding to the same underlying bracketings:
\begin{align*}
\rho(a_1a_2,a_3,a_4)&=\Big[1234\Big]_1 + \Big[- x_1 1234 - x_2 1243\Big]_3 + \Big[- x_3 3124 - x_5 4123\Big]_4 
+ \Big[ - x_4 3412  - x_6 4312\Big]_5
\\
{}
\rho(a_1,a_2a_3,a_4)&= \Big[1234\Big]_2 +\Big[- x_3 2314 - x_4 2341\Big]_3 + \Big[- x_1 1234 - x_6 4231\Big]_4 
+ \Big[- x_2 1423  - x_5 4123\Big]_5
\\
{}
\rho(a_1,a_2,a_3a_4)&=\Big[1234  - x_5 3412 - x_6 3421\Big]_3  + \Big[-x_2 1342 - x_4 2341\Big]_4 
+ \Big[-x_1 1234 - x_3 2134\Big]_5
\\
{}
\rho(a_1,a_2,a_3)a_4&=\Big[1234\Big]_1 + \Big[- x_1 1234 - x_2 1324 - x_3 2134 - x_4 2314 
- x_5 3124 - x_6 3214\Big]_2
\\
{}
a_1\rho(a_2,a_3,a_4)&=\Big[1234\Big]_4 + \Big[- x_1 1234 - x_2 1243 - x_3 1324 - x_4 1342 
- x_5 1423 - x_6 1432\Big]_5
\end{align*}
Each pair of square brackets in each of these expressions contains an element of the group algebra
$\mathbb{F} S_4$ with coefficients extended to the polynomial ring $\mathbb{F}[x_1,\dots,x_6]$.

We now apply equation \eqref{Gmatrix} to compute the representation matrices $R_\lambda(\calG)$ of
the relations $\calG$ in each partition $\lambda$.
In this way we replace the original $120 \times 120$ cubic relation matrix $M$ by five smaller matrices
of sizes $5 d_\lambda \times 5 d_\lambda$ for $d_\lambda = 1, 3, 2, 3, 1$.
Regularity holds if and only if the nullity of $R_\lambda(\calG)$ equals $d_\lambda$ for all $\lambda$.
This guarantees that $\calT(4)$ contains exactly $d_\lambda$ copies of the simple module $[\lambda]$,
and is therefore isomorphic to the regular module $\mathbb{F} S_4$.
Equivalently, the rank of $R_\lambda(\calG)$ must equal $4 d_\lambda$ for all $\lambda$; in terms of
determinantal ideals, this means that for $d = d_\lambda$ we have $DI_{4d}(R_\lambda(\calG)) \ne \{0\}$
and $DI_{4d+1}(R_\lambda(\calG)) = \{0\}$ for all $\lambda$.
This proves the following result.

\begin{lemma} \label{propertyRrankconditions}
	The parametrized one-relation operad is regular in arity $4$ for particular values of the parameters $x_1,\dots,x_6$
	if and only if all of the following conditions hold:
	\[
	\begin{array}{lll}
	\lambda = 4\colon \qquad  &
	DI_4(R_\lambda(\calG)) \ne \{0\},  &
	DI_5(R_\lambda(\calG)) = \{0\}
	\\
	\lambda = 31\colon  &
	DI_{12}(R_\lambda(\calG)) \ne \{0\},  &
	DI_{13}(R_\lambda(\calG)) = \{0\}
	\\
	\lambda = 2^2\colon  &
	DI_8(R_\lambda(\calG)) \ne \{0\},  &
	DI_9(R_\lambda(\calG)) = \{0\}
	\\
	\lambda = 21^2\colon  &
	DI_{12}(R_\lambda(\calG)) \ne \{0\},  &
	DI_{13}(R_\lambda(\calG)) = \{0\}
	\\
	\lambda = 1^4\colon  &
	DI_4(R_\lambda(\calG)) \ne \{0\},  &
	DI_5(R_\lambda(\calG)) = \{0\}
	\end{array}
	\]
\end{lemma}

\begin{remark}
The conditions in Lemma \ref{propertyRrankconditions} may be combined and simplified.
Let $G_1,\dots,G_k$ be Gr\"obner bases for ideals $I_1,\dots,I_k \subseteq \mathbb{F}[x_1,\dots,x_d]$.
	Consider these two equations:
	\[
	V( \, I_1 + \cdots + I_k \, ) = V( I_1 ) \,\cap\, \cdots \,\cap\, V( I_k ),
	\qquad \qquad
	V\big(\sqrt{I}\big) = V(I) = \bigcap_{i=1}^k V(I_i).
	\]
	From the generating set $G = G_1 \cup \cdots \cup G_k$ for the ideal $I = I_1 + \cdots + I_k$ we compute a
	Gr\"obner basis $H$, and from this we compute a Gr\"obner basis $K$ for the radical $\sqrt{I}$.
	We solve the system of equations $\{ \, f = 0 \mid f \in K \, \}$ to find $V\big(\sqrt{I}\big)$.
	To include the lower rank conditions $DI_{4d}(R_\lambda(\calG)) \ne \{0\}$, we substitute each solution
	into the Gr\"obner bases for the lower ideals $DI_{4d}(R_\lambda(\calG))$, and retain a solution
	if and only if it is not in $Z(DI_{4d}(R_\lambda(\calG)))$ for any $\lambda$.
\end{remark}


We noted in Remark \ref{hugenumbers} that if the number of minors is too large then it is not practical
to compute a Gr\"obner basis for a determinantal ideal.
Using representation theory allows us to go much further.
To apply Lemma \ref{propertyRrankconditions}, we need to compute
\begin{itemize}
	\item
	all minors of sizes 4 and 5 for the $5 \times 5$ matrix $R_\lambda(\calG)$ when $\lambda = 4$ and
	$\lambda = 1^4$
	\item
	all minors of sizes 12 and 13 for the $15 \times 15$ matrix $R_\lambda(\calG)$ when $\lambda = 31$
	and $\lambda = 21^2$
	\item
	all minors of sizes 8 and 9 for the $10 \times 10$ matrix $R_\lambda(\calG)$ when $\lambda = 2^2$
\end{itemize}
The total is extremely small compared to the numbers in Remark \ref{hugenumbers}:
\[
2 \left[ \binom{5}{4}^2 + \binom{5}{5}^2 \right] +
2 \left[ \binom{15}{12}^2 + \binom{15}{13}^2 \right] +
\binom{10}{8}^2 + \binom{10}{9}^2 = 438277.
\]
Notably, those matrices have many zero entries; furthermore, they have entries which are nonzero scalars ($\pm 1$), so we can apply Algorithm \ref{partialsmithform} to reduce their sizes even further, as we did in Section \ref{cubicrelationmatrixsection} when we extracted the $36 \times 36$ block $B$
from the cubic relation matrix $M$.

\subsection{Reduction of the representation matrices}

The representation matrices $R_\lambda(\calG)$ are square of sizes 5, 15, 10, 15, 5 respectively. 
Applying Algorithm \ref{partialsmithform} reduces each of these to a block diagonal matrix $[I,B]$ of where $I$
(identity matrix) and $B$ (block with no nonzero scalars) have sizes $r$ and $s$ respectively, where
$[r,s]$ is one of the pairs $[3,2]$, $[10,5]$, $[6,4]$, $[10,5]$, and $[3,2]$.
We write $B(\lambda)$ for the block corresponding to partition $\lambda$.
If $B(\lambda)$ is $s \times s$ then $DI_s(B(\lambda))$ is the principal ideal generated by
$\det(B(\lambda))$, and so $DI_s(B(\lambda)) = \{0\}$ if and only if $\det(B(\lambda)) = 0$.
The next result is Lemma \ref{propertyRrankconditions} reformulated in terms of the reduced matrices $B(\lambda)$.

\begin{lemma} \label{reducedpropertyRconditions}
	Regularity holds for particular values of the parameters $x_1,\dots,x_6$ if and only if the following
	conditions on the determinantal ideals of $B(\lambda)$ hold for all $\lambda$:
	\[
	\begin{array}{lccl}
	\lambda &\qquad B(\lambda) &\qquad DI_r(B(\lambda)) \ne \{0\} &\qquad DI_{r+1}(B(\lambda)) = \{0\}
	\\ \midrule
	4  &\qquad  2 \times 2  &\qquad  r = 1  &\qquad  r+1 = 2, \quad \det(B(\lambda)) = 0
	\\
	31  &\qquad  5 \times 5  &\qquad  r = 2  &\qquad  r+1 = 3
	\\
	2^2  &\qquad  4 \times 4  &\qquad  r = 2  &\qquad  r+1 = 3
	\\
	21^2  &\qquad  5 \times 5  &\qquad  r = 2  &\qquad  r+1 = 3
	\\
	1^4  &\qquad  2 \times 2  &\qquad  r = 1  &\qquad  r+1 = 2, \quad \det(B(\lambda)) = 0
	\end{array}
	\]
\end{lemma}


\section{Main technical result}
\label{proofclassificationsection}

In this section we describe the computations which allow us to complete the classification of parametrized
one-relation operads for which the arity $4$ component is the regular $S_4$-module. These computations are based on the reduced representation matrices $B(\lambda)$  collated in the online addendum to this paper \cite{BDadd}. Essentially the same methods can be used to determine all instances of Relation \eqref{LRrelation} which produce any desired $S_4$-module structure in arity $4$, not necessarily the regular one.

We increase the complexity of the problem step by step, starting with the case of one nonzero parameter,
and ending with the general case in which all six parameters all allowed to be nonzero.
In order to avoid linguistic pedantry, when we say that the parameters in some subset
$S \subseteq P = \{x_1,x_2,x_3,x_4,x_5,x_6\}$ are nonzero, we mean that we are setting the parameters
in $P\setminus S$ to zero and regarding those in $S$ as free.

We call the ideals $DI_{4d_\lambda+1}(R_\lambda(\calG))$ \emph{upper determinantal ideals}, and the ideals $DI_{4d_\lambda}(R_\lambda(\calG))$ \emph{lower determinantal ideals}; according to Lemma \ref{propertyRrankconditions}, for a parametrized one-relation operad to be regular in arity $4$, the set of parameters must be a common zero of all upper determinantal ideals, and must be outside the zero set of each lower determinantal ideal.  
We denote by the symbols $\Sigma+$ and  $\sqrt{\Sigma +}$ the sum of the upper determinantal ideals and its radical
respectively.


\subsection{One nonzero parameter}

When the only nonzero parameter is $x_1$, for every representation $[\lambda]$ the upper ideal is generated by
$x_1^2(x_1-1)$ and the lower ideal is generated by 1.
Then clearly the sum of the upper ideals is generated by $x_1^2(x_1-1)$ and its radical is generated by $x_1(x_1-1)$.
For regularity, the sum of the upper ideals must be $\{0\}$, giving $x_1 = 0$ or $x_1 = 1$,
and each lower ideal must be nonzero (which is clear).
The solution $x_1 = 0$ corresponds to the left-nilpotent identity $(a_1a_2)a_3 = 0$, and $x_1 = 1$ corresponds to
associativity $(a_1a_2)a_3 = a_1(a_2a_3)$.

When the only nonzero parameter is $x_2$, $x_3$ or $x_4$, the only regular solution is the zero solution
(left-nilpotent identity).

When the only nonzero parameter is $x_5$, for every representation $[\lambda]$ the upper ideal is zero, and the lower ideals are 
generated by 
 \[
x_5-1,\quad  (x_5-1)(x_5+1)^2, \quad (x_5-1)^2, \quad  (x_5-1)(x_5+1)^2, \quad x_5-1, 
 \]
and we will have a regular solution if and only if every lower ideal is nonzero, and this happens if and
only if $x_5 \ne \pm 1$.

When the only nonzero parameter is $x_6$, the upper ideals are generated by
\[
x_6(x_6+1)(x_6-1)^2, \quad
x_6(x_6-1)(x_6+1)^2, \quad
x_6(x_6-1)^2(x_6+1)^2, \quad
x_6(x_6+1)(x_6-1)^2,\quad
x_6(x_6-1)(x_6+1)^2,
\]
and the radical of their sum  consists of all multiples
of $x_6(x_6-1)(x_6+1)$ and hence will be zero if and only if $x_6 \in \{ 0, \pm 1 \}$.
The lower ideal are generated by 
 \[
x_6+1,\quad x_6-1, \quad x_6^2-1, \quad x_6-1, x_6+1, 
 \]
and the only one of these values which does not make at least one lower ideal equal
to zero is $x_6 = 0$, and so here again we recover only the zero solution.

\begin{proposition}[Summary for (at most) one nonzero parameter]\label{prop:1param}
	When at most one of the parameters in Relation \eqref{LRrelation} is nonzero, there are three solutions giving the regular module in arity $4$; two isolated and
	one 1-dimensional (a one-parameter family):
	\[
	(a_1a_2)a_3 = 0, \quad
	(a_1a_2)a_3 = a_1(a_2a_3), \quad
	(a_1a_2)a_3 = x_5 a_3(a_1a_2) \quad (x_5 \ne \pm 1).
	\]	
\end{proposition}


\subsection{Two nonzero parameters}

From now on, ideals are not necessarily principal, so Gr\"obner bases typically contain two or more elements.
There are 15 cases when we choose two parameters from six, but it will not be necessary to discuss all of
them in detail.
We begin with $x_1,x_2$ and continue in lex order.

\subsubsection*{$x_1, x_2$ nonzero}

The upper determinantal ideals have the following \texttt{grevlex} Gr\"obner bases:
\begin{alignat*}{1}
DI^+_4 &= \big( \; (x_2-x_1)(x_2+x_1)(x_2+x_1-1) \; \big)
\\[-2pt]
DI^+_{31} &= \big( \; x_1^3+x_2^2-x_2x_1-x_1^2, \; x_2(x_1^2-x_2), \; x_2x_1(x_2-1), \; x_2^2(x_2-1) \; \big)
\\[-2pt]
DI^+_{2^2} &= \big( \; x_1^3+x_2^2-x_2x_1-x_1^2, \; x_2(x_1^2-x_2), \; x_2x_1(x_2-1), \; x_2^2(x_2-1) \; \big)
\\[-2pt]
DI^+_{21^2} &= \big( \; x_2(x_2-x_1), \; x_1^2(x_1-1), \; x_2x_1(x_1-1) \; \big)
\\[-2pt]
DI^+_{1^4} &= \big( \; (x_2-x_1+1)(x_2-x_1)^2 \; \big)
\end{alignat*}
The sum of these ideals is the ideal $\Sigma+$ for $\lambda = 21^2$, and its radical has this Gr\"obner basis
and zero set:
\[ 
\sqrt{\Sigma+} = \big( \; x_1(x_1-1), \; x_2(x_1-1), \; x_2(x_2-1) \; \big), \qquad 
V\big(\sqrt{\Sigma+}\big) = \big\{ \; (x_1,x_2) = (0,0), \; (1,0), \; (1,1) \; \big\}
\]
Every lower ideal has Gr\"obner basis $\{1\}$, so all three of the solutions are regular.
We have already seen the first and second, but the third is new: it defines the Zinbiel identity
 \[
(a_1a_2)a_3 = a_1(a_2a_3) + a_1(a_3a_2). 
 \]

\subsubsection*{$x_1, x_3$ nonzero}

This is the Koszul dual of the case $x_1, x_2$ nonzero.
To derive the results in this case from those of the previous case, for each $\lambda$ we replace $x_2$ by $-x_3$
and $\lambda$ by its conjugate; this corresponds to tensoring with the sign module.
We obtain again the trivial and associative identities, since they are self-dual, but the Zinbiel identity is
transformed into the Leibniz identity: $(a_1a_2)a_3 = a_1(a_2a_3) - a_2(a_1a_3)$.

\subsubsection*{$x_1, x_4$ nonzero}

The radical of the sum of the upper ideals is generated by the polynomials $x_4$ and $x_1(x_1-1)$, so $x_4=0$, and there are no new solutions.

\subsubsection*{$x_1, x_5$ nonzero}

The radical of the sum of the upper ideals is
\[
\sqrt{\Sigma+} = \big( \; x_1(x_1-1), \; x_5x_1 \; \big),
\]
so $x_1x_5=0$, and there are no new solutions.

\subsubsection*{$x_1, x_6$ nonzero}

The radical of the sum of the upper ideals is
\[
\sqrt{\Sigma+} = \big( \; x_1(x_1-1), \; x_6x_1, \; x_6(x_6-1)(x_6+1) \; \big),
\]
so $x_1x_6=0$, and there are no new solutions.

\subsubsection*{$x_2, x_3$ nonzero}

We have $\sqrt{\Sigma+} = \big( x_2, \; x_3 \big)$, so there are no new solutions.

\subsubsection*{$x_2, x_4$ nonzero to $x_4, x_6$ nonzero}

No new features; we omit the details.

\subsubsection*{$x_5, x_6$ nonzero}

The radical of the sum of the upper ideals is
\[
\sqrt{\Sigma+} = \big( \; x_6x_5, \; x_6(x_6-1)(x_6+1) \; \big),
\]
so $x_5x_6=0$, and there are no new solutions.

\begin{proposition}[Summary for two nonzero parameters]\label{prop:2param}
	When exactly two parameters in Relation \eqref{LRrelation} are different from zero,  there are two regular solutions, both isolated: the Zinbiel and Leibniz identities:
	\[
	(a_1a_2)a_3 = a_1(a_2a_3) + a(a_3a_2), \qquad \qquad (a_1a_2)a_3 = a_1(a_2a_3) - a_2(a_1a_3).
	\]	
\end{proposition}


\subsection{Three nonzero parameters}

There are $\binom63 = 20$ cases, starting with $x_1$, $x_2$, $x_3$ in lex order and ending with $x_4$, $x_5$, $x_6$, but they produce
no new regular solutions.
We present details only for the first and last cases, since they illustrate the computations that are typical
of all cases.

\subsubsection*{$x_1,x_2,x_3$ nonzero}

Once we compute Gr\"obner bases for the radicals of the upper ideals, we in particular note that 
 \[
\sqrt{DI^+_{31}} = \big( \; (x_1-1)(x_2+x_1), \; x_3(x_1-1), \; x_2(x_2-1), \; x_3x_2, \; x_3(x_3+1), \; x_1(x_1-1)(x_1+1) \; \big) .
 \]
We see that $x_2x_3=0$, so there are no new solutions.

\subsubsection*{$x_1,x_2,x_4$ to $x_3,x_5,x_6$ nonzero}

No new features; we omit the details.

\subsubsection*{$x_4,x_5,x_6$ nonzero}

Once we compute Gr\"obner bases for the radicals of the upper ideals, we in particular note that 
 \[
\sqrt{DI^+_{31}} =
\big( \; x_4, \; x_6x_5(x_6+x_5+1), \; x_6(x_6+x_5+1)(x_6-x_5-1) \; \big) .
 \]
We see that $x_4 = 0$, so there are no new solutions.

\begin{proposition}[Summary for three nonzero parameters]\label{prop:3param}
	When exactly three parameters in Relation \eqref{LRrelation} are nonzero,  there are no solutions which are regular in arity $4$.
\end{proposition}


\subsection{Four nonzero parameters}

In this case, we obtain two new relations with irrational coefficients which are regular; but we will
see shortly that these solutions belong to a one-parameter family, all of whose other solutions have five
nonzero coefficients.
We discuss these two cases, $x_2,x_4,x_5,x_6$ nonzero and $x_3,x_4,x_5,x_6$ nonzero, and one other case,
$x_2,x_3,x_4,x_6$ nonzero, which is remarkable for the complexity of the Gr\"obner bases that occur.

\subsubsection*{$x_2,x_3,x_4,x_6$ nonzero}

The individual upper ideals have very complicated Gr\"obner bases with dozens of terms some of which have coefficients of absolute value about $10^{23}$. However, when we consider the sum of the upper ideals, the complexity vanishes: 
the Gr\"obner basis for the sum contains only 7 polynomials of degrees 1,2,3 with 1 or 2 terms and
all coefficients $\pm 1$.
The radical is slightly simpler: only 4 polynomials, and none of degree 2:
\[
\sqrt{\Sigma +}
 = \big( \; x_2, \; x_3, \; x_4, \; x_6(x_6+1)(x_6-1) \; \big) .
\]
We see that $x_2=x_3=x_4=0$, so there are no new solutions.

\subsubsection*{$x_2,x_4,x_5,x_6$ nonzero}

In this case the radical $\sqrt{\Sigma +}$ has the following Gr\"obner basis:
\[
 x_4+x_2, \quad x_2(x_5+x_2), \quad x_2(x_6-1), \quad x_5x_6+x_2, \quad x_2(x_2^2-x_2-1), \quad x_6(x_6-1)(x_6+1) 
\]
We assume $x_2 \ne 0$, so we may cancel the factor $x_2$ from three generators, obtaining
\[
\{ \; x_4+x_2, \; x_5+x_2, \; x_6-1, \; x_5x_6+x_2, \; x_2^2-x_2-1, \; x_6(x_6-1)(x_6+1) \; \}.
\]
If we set $x_4 = -x_2$, $x_5 = -x_2$, $x_6 = 1$ then this generating set reduces to $\{ \, x_2^2-x_2-1 \, \}$. Therefore, we obtain solutions
\begin{equation}
\label{golden}
[x_1,x_2,x_3,x_4,x_5,x_6] \;=\; [ 0, \phi, 0, -\phi, -\phi, 1 ]
\end{equation}
where $\phi$ can be either of the roots of the polynomial $x^2-x-1$.
Before we can verify that this is regular, we must consider the lower ideals, whose radicals are:
\[
\begin{array}{ll}
\lambda = 4    &\qquad \big( \; x_2, \; x_6+x_5-1, \; x_4(x_4+1), x_5x_4 \; \big) \\
\lambda = 31   &\qquad \big( \;x_2, \; x_4, \; x_5^2-x_6-1, \; x_6x_5, \; x_6(x_6+1) \; \big) \\
\lambda = 2^2  &\qquad \big( \; x_2, \; x_4, \; x_5(x_5-1), \; x_6x_5, \; x_6^2+x_5-1 \; \big) \\
\lambda = 21^2 &\qquad \big( \; x_2, \; x_4, \; x_5^2+x_6-1, \; x_6x_5, \; x_6(x_6-1) \; \big) \\
\lambda = 1^4  &\qquad \big( \; x_6-x_5+1, \; x_2(x_4-x_2), \; (x_4-x_2)(x_4+x_2+1), \; x_5(x_4-x_2) \; \big)
\end{array}
\]
For parameters equal to the values \eqref{golden}, some of these polynomials do not vanish: the first four ideals contain $x_2=\phi \ne 0$, and the fifth contains
$x_6-x_5+1 = \phi+2 \ne 0$.

\subsubsection*{$x_3,x_4,x_5,x_6$ nonzero}

The calculations are similar to those of the previous case, and we obtain two new
solutions:
\begin{equation}
\label{golden2}
[x_1,x_2,x_3,x_4,x_5,x_6] \;=\; [ 0, 0, -\phi, -\phi, -\phi, -1 ] .
\end{equation}

\begin{proposition}[Summary for four nonzero parameters]\label{prop:4param}
	When exactly four parameters in Relation \eqref{LRrelation} are different from zero,  there are four solutions which are regular in arity $4$, two for each root $\phi$ of the polynomial $x^2-x-1$:
	\begin{align*}
	(a_1a_2)a_3 &=  \phi \, a_1(a_3a_2) - \phi \, a_2(a_3a_1) - \phi \, a_3(a_1a_2) + \, a_3(a_2a_1),
	\\
	(a_1a_2)a_3 &= -\phi \, a_2(a_1a_3) - \phi \, a_2(a_3a_1) - \phi \, a_3(a_1a_2) - \, a_3(a_2a_1) .	
	\end{align*}
\end{proposition}


\subsection{Five nonzero parameters}

We obtain a new one-parameter family involving the first five parameters,
We present details of the computations in this case, and omit the others which do not
produce any new solutions.

\subsubsection*{$x_1,x_2,x_3,x_4,x_5$ nonzero}

Although the individual upper ideals have very complicated Gr\"obner bases with hundreds of terms some of which have coefficients of absolute value about $10^{15}$, the radical $\sqrt{\Sigma+}$ has this simple Gr\"obner basis:
\begin{equation}
\label{uppersum5}
\begin{array}{ll}
(x_3+x_2)(x_1-1), &\qquad (x_1-1)(x_4+x_1), \\
x_3x_2+x_2^2-x_5x_1-x_2, &\qquad x_4x_2-x_2^2+x_2x_1+x_1^2-x_4-x_1, \\
(x_2-x_1)(x_5+x_2+x_1-1), &\qquad (x_3+x_2)(x_3-x_2+1), \\
x_4x_3+x_2^2-x_2x_1-x_1^2+x_4+x_1, &\qquad x_5x_3-x_2^2+x_5x_1+x_1^2+x_2-x_1, \\
x_4^2-x_2^2+x_5x_1+x_2, &\qquad x_5x_4+x_2^2-x_1^2+x_4-x_2+x_1, \\
(x_1-1)(x_5x_1+x_2), &\qquad (x_1-1)(x_2-x_1)(x_2+x_1), \\
x_5^2x_1-x_2^2+x_5x_1+x_1^2+x_2-x_1 . 
\end{array}
\end{equation}
We note that several of these polynomials are divisible by $x_1-1$, so may use a divide-and-conquer strategy to find the zero set of these polynomials.

\noindent\emph{Case 1}. Setting $x_1 = 1$ in the polynomials \eqref{uppersum5} and recomputing the Gr\"obner basis produces
these 9 polynomials:
\begin{equation}
\label{case1gb}
\begin{array}{lll}
x_2^2+x_2x_3-x_2-x_5, &\qquad (x_2-1)(x_4-x_2), &\qquad (x_2-1)(x_5+x_2), \\
(x_3+x_2)(x_3-x_2+1), &\qquad x_3x_4+x_2^2-x_2+x_4, &\qquad x_3x_5-x_2^2+x_2+x_5, \\
x_4^2-x_2^2+x_2+x_5, &\qquad x_4x_5+x_2^2-x_2+x_4, &\qquad (x_2+x_5)(x_5-x_2+1) .
\end{array}
\end{equation}
We note that two of the polynomials are divisible by $x_2-1$, so may use a divide-and-conquer strategy again.

\noindent\emph{Subcase 1a}. Setting $x_2 = 1$ and recomputing the Gr\"obner basis produces
\[
x_5-x_3, \qquad x_3(x_3+1), \qquad (x_3+1)x_4, \qquad x_4^2+x_3 .
\]
Since $x_3\ne 0$, we have $x_3 = -1$, so that $x_4 = \pm 1$ and $x_5 = -1$, giving the solutions
\[
[x_1,x_2,x_3,x_4,x_5,x_6] \; = \; [ 1, 1, -1, 1, -1, 0 ], \quad [ 1, 1, -1, -1, -1, 0 ] .
\]

\noindent\emph{Subcase 1b}. If $x_2 \ne 1$ then we may divide by (= remove) the two factors $x_2-1$ in the polynomials
\eqref{case1gb} and recompute the Gr\"obner basis, obtaining
\[
x_4-x_2, \qquad x_2+x_5, \qquad x_2(x_3+x_2), \qquad (x_3+x_2)(x_3-x_2+1) .
\]
Since $x_2\ne 0$, we have $x_3 = -x_2$, so that $x_4 = x_2$ and $x_5 = -x_2$, giving the solution
\[
[x_1,x_2,x_3,x_4,x_5,x_6] \; = \; [ 1, x_2, -x_2, x_2, -x_2, 0 ] \quad ( x_2 \ne 1 ) .
\]

\noindent\emph{Case 2}. If $x_1 \ne 1$, then we can remove the factors $x_1-1$ from the polynomials \eqref{uppersum5}
and recompute the Gr\"obner basis, obtaining
\begin{equation}
\label{case2gb}
x_3+x_2, \quad x_4+x_1, \quad x_1x_5+x_2, \quad (x_2-x_1)(x_2+x_1), \quad x_2x_5+x_1 .
\end{equation}
If $x_2 = x_1$, then \eqref{case2gb} reduces to $\{ \; x_3+x_1, \; x_4+x_1, \; x_1(x_5+1) \; \}$ and so we have one new solution
\[
[x_1,x_2,x_3,x_4,x_5,x_6] \; = \;
[ x_1, x_1, -x_1, -x_1, -1, 0 ] \quad (x_1 \ne 0,1) .
\]
If $x_2 \ne x_1$, then \eqref{case2gb} reduces to $\{ \; x_3-x_1, \; x_4+x_1, \; x_1(x_5-1) \; \}$ and so we have one new solution
\[
[x_1,x_2,x_3,x_4,x_5,x_6] \; = \;
[ x_1, -x_1, x_1, -x_1, 1, 0 ] \quad (x_1 \ne 0,1) .
\]
We sort the complete list of solutions by increasing number of nonzero parameters:
\begin{equation}
\label{solutions5}
\begin{array}{rll}
\# &\qquad [x_1,x_2,x_3,x_4,x_5,x_6] &\qquad \text{comments} \\ \midrule
1 &\qquad [ 1,  x_2, -x_2,  x_2, -x_2, 0 ] &\qquad \text{including $x_2 = 1$} \\
2 &\qquad [ x_1, -x_1,  x_1, -x_1,  1, 0 ] &\qquad x_1 \ne 0, 1 \\
3 &\qquad [ x_1,  x_1, -x_1, -x_1, -1, 0 ] &\qquad x_1 \ne 0, \; \text{including $x_1 = 1$}
\end{array}
\end{equation}
The solutions $[ 1,  1, -1,  1, -1, 0 ]$ and  $[ 1,  1, -1, -1, -1, 0 ]$ now become special cases of \#1 and \#3 respectively.
It is easy to verify by direct substitution that all these solutions belong to the zero set of
every polynomial in the
Gr\"obner basis \eqref{uppersum5}.

To determine which of the solutions \eqref{solutions5} are regular, we need to look at the lower
ideals for the five partitions.
Their radicals have the following Gr\"obner bases:
\begin{alignat*}{1}
DI^-_4&\;\;\;\;
x_5-1, \; x_4x_1+x_3x_1+x_2x_1+x_1^2-x_4-x_3\\[-1mm]
&\quad
(x_2+x_1)(x_3-x_2+2x_1-2) \;  x_4x_2+x_2^2-x_3x_1-x_1^2+x_4+x_3+x_2+x_1, \\[-1mm]
&\quad
x_4x_3+x_3^2-x_2x_1-x_1^2+x_4+x_3-x_2-x_1, \; x_4^2-x_3^2-x_4-x_3+2x_2+2x_1,  \\[-1mm]
&\quad
(x_2+x_1)(x_2x_1-x_1^2+3x_1-1)\\[-1mm]
DI^-_{31} &\;\;\;\;
x_4+x_3+x_2+x_1, \;  (x_1+1)(x_3+x_2), \;  x_5x_1-x_3,  \\[-1mm]
&\quad
(x_2-x_1)(x_2+x_1),  \;  x_3x_2+x_1^2-x_3-x_2,  \;  x_5x_2+x_3+x_2+x_1,  \\[-1mm]
&\quad
(x_3-x_1)(x_3+x_1),  \; x_5x_3-x_1,  \;  (x_5-1)(x_5+1)\\[-1mm]
DI^-_{2^2}&\;\;\;\;
x_3-x_2, \;  x_4-x_1,\;  x_5-1,\;  x_1(x_1-1), \;  x_2(x_1-1), \;  x_2^2-x_1 \\[-1mm]
DI^-_{21^2}&\;\;\;\;
x_4-x_3-x_2+x_1,\;  (x_1+1)(x_3+x_2),\;  x_5x_1+x_2,  \\[-1mm]
&\quad
(x_2-x_1)(x_2+x_1),\;  x_3x_2+x_1^2+x_3+x_2,\;  x_5x_2+x_1,  \\[-1mm]
&\quad
(x_3-x_1)(x_3+x_1),\;   x_5x_3+x_3+x_2-x_1,\;   (x_5-1)(x_5+1)\\[-1mm]
DI^-_{1^4}&\;\;\;\;
x_5-1, \; x_4x_1-x_3x_1-x_2x_1+x_1^2-x_4+x_2,  \\[-1mm]
&\quad 
x_4x_2-x_2^2-x_3x_1+x_1^2-x_4-x_3+x_2+x_1, \;  (x_3-x_1)(x_3-x_2+2x_1-2),  \\[-1mm]
&\quad 
x_4x_3-x_3x_2+x_3x_1-x_1^2-x_4-x_3+x_2+x_1,\\[-1mm] 
&\quad x_4^2-x_2^2-x_4-2x_3+x_2+2x_1, \; 
(x_3-x_1)(x_2x_1-x_1^2-x_1+1)
\end{alignat*}
We substitute the three solutions \eqref{solutions5} into these Gr\"obner bases, which makes all the polynomials univariate, and
determine the ideal these univariate polynomials generate. 
For each of the solutions, we obtain a list of five ideals corresponding to the five partitions, and each ideal
must be nonzero in order for regularity to hold:
\[
\begin{array}{l|ccccc}
[x_1,x_2,x_3,x_4,x_5,x_6] & \lambda = 4 &\lambda = 31 &\lambda = 2^2 &\lambda =21^2 &\lambda =1^4
\\
\midrule
{}
[1,x_2,-x_2,x_2,-x_2,0] &
\big( \, x_2+1 \big) &
\big( \, x_2+1 \big) &
\big( \, 1 \big) &
\big( \, x_2+1 \big) &
\big( \, x_2+1 \big)
\\
{}
[x_1,-x_1,x_1,-x_1,1,0] &
\big(  0 \big) &
\big(  0 \big) &
\big(  x_1 \big) &
\big(  0 \big) &
\big(  0 \big)
\\
{}
[x_1,x_1,-x_1,-x_1,-1,0] &
\big( 1 \big) &
\big(  0 \big) &
\big(  1 \big) &
\big(  0 \big) &
\big(  1 \big)
\end{array}
\]
Thus the solution $[1,x_2,-x_2,x_2,-x_2,0]$ with $x_2 \ne -1$ is the only regular one.

\begin{proposition}[Summary for five nonzero parameters]\label{prop:5param}
	When exactly five parameters in Relation \eqref{LRrelation} are nonzero,  there is a one-dimensional family of solutions which are regular in arity $4$:
	\[
	(a_1a_2)a_3=a_1(a_2a_3) + x_2\big[a_1(a_3a_2) - a_2(a_1a_3) + a_2(a_3a_1) -a_3(a_1a_2)\big]  \quad (x_2 \ne -1) .
	\]
\end{proposition}

\begin{remark}\label{rem:5paramexceptional}
	The exceptional case $x_2 = -1$ gives a relation that is not regular:
	\[
	(a_1a_2)a_3 = a_1(a_2a_3) - a_1(a_3a_2) + a_2(a_1a_3) - a_2(a_3a_1) + a_3(a_1a_2).
	\]
	In this case, the cubic relation matrix $M$ has nullity 32 and multiplicities [2,4,2,4,2]; that is, the
	nullspace is isomorphic to the $S_4$-module $2[4] \oplus 4[31] \oplus 2[2^2] \oplus 4[21^2] \oplus 2[1^4]$.
\end{remark}


\subsection{Six nonzero parameters}

There is only one case (all parameters are free), and we may assume that each parameter is nonzero, since
if any parameter is zero then we return to one of the cases already considered.

\subsubsection*{Upper ideals}

The sum $\Sigma+$ of the upper ideals has a \texttt{grevlex} Gr\"obner basis consisting of
83 elements, degrees 3 to 5, terms 2 to 117, and coefficients $-1642727092$ to $1636813156$.
There are 5 elements of degree 3, 62 of degree 4, and 16 of degree 5.
Exactly two elements (numbers 2 and 5) have a parameter as an irreducible factor,
in both cases $x_6$:
\[
g_2 = -x_6(x_1x_4-x_2x_3), \qquad\qquad g_5 = x_6(x_1^2-x_2^2-x_3^2+x_4^2).
\]
Since $x_6 \ne 0$ by assumption, we may divide both $g_2$ and $g_5$ by $x_6$ and replace them in the Gr\"obner basis by
\[
g'_2 = -x_1x_4+x_2x_3, \qquad\qquad g'_5 = x_1^2-x_2^2-x_3^2+x_4^2.
\]
We recompute the Gr\"obner basis and obtain 65 elements with degrees 2 to 5, terms 2 to 91, and coefficients
$-3024000276$ to $2254275346$.
There are 2 elements of degree 2, 1 of degree 3, 49 of degree 4, and 13 of degree 5.
The first two elements of this new basis are $g'_2$ and $g'_5$.
We call this the \emph{simplified upper basis}.

Since $x_1 \ne 0$ by assumption, we solve for $x_4$ in $g'_2 = 0$ and obtain $x_4 = x_2x_3/x_1$.
We substitute this in $g'_5$ and factor the result, obtaining
\[
\frac{(x_1-x_2)(x_1+x_2)(x_1-x_3)(x_1+x_3)}{x_1^2} .
\]
For every solution, this must vanish, so we may split the computation
of the zero set of the simplified upper basis into four cases:
\[
x_2 = x_1, \qquad x_2 = -x_1, \qquad x_3 = x_1, \qquad x_3 = -x_1.
\]
Making these substitutions into $g'_2$ we obtain
\[
x_1(x_3-x_4), \qquad -x_1(x_3+x_4), \qquad x_1(x_2-x_4), \qquad -x_1(x_2+x_4).
\]
Since $x_1 \ne 0$, in each case the other factor is 0, and so the four cases are defined as follows:
\begin{equation}
\label{fourcases}
\begin{array}{c@{\qquad}|l@{\qquad}l@{\qquad}|l@{\qquad}}
\text{case} &\multicolumn{2}{l|}{\text{substitutions}} & \text{relation coefficients} \\
1 & x_2 =  x_1, & x_4 =  x_3 & [ \, x_1, \,  x_1, \,  x_3, \,  x_3, \, x_5, \, x_6 \, ] \\
2 & x_2 = -x_1, & x_4 = -x_3 & [ \, x_1, \, -x_1, \,  x_3, \, -x_3, \, x_5, \, x_6 \, ] \\
3 & x_3 =  x_1, & x_4 =  x_2 & [ \, x_1, \,  x_2, \,  x_1, \,  x_2, \, x_5, \, x_6 \, ] \\
4 & x_3 = -x_1, & x_4 = -x_2 & [ \, x_1, \,  x_2, \, -x_1, \, -x_2, \, x_5, \, x_6 \, ]
\end{array}
\end{equation}
In this way we reduce the original problem with 6 free parameters to four much smaller problems each with
4 free parameters.

For each of these cases, we make the corresponding substitutions into the simplified upper basis,
and recompute the Gr\"obner basis.
We then repeatedly cancel irreducible factors in basis elements which are parameters, and recompute the Gr\"obner basis.

All these tricks seem necessary to be able to compute a Gr\"obner basis for the radical of the
sum of the upper ideals in a reasonable time. We obtain the following results:

\smallskip

\noindent \emph{Case 1}. The original basis of 65 elements reduces to 26, 21, 12 elements after cancelling $x_1$ five times;
the resulting basis has 2, 4, 4, 2 elements of degrees 2, 3, 4, 5 respectively, terms from 9 to 34, and
coefficients from $-249$ to $211$.
The radical of this ideal has the following Gr\"obner basis:
\begin{gather*}
x_6+x_5-x_3-x_1+1, \quad (x_5-x_3)(2x_5-x_3-x_1+2),\\
2x_3^2x_1-2x_1^3+x_5x_3-x_3^2-x_5x_1+3x_3x_1+4x_1^2-2x_1,\\
6x_5x_3x_1+6x_5x_1^2-6x_3x_1^2-6x_1^3+x_5x_3-x_3^2-5x_5x_1+11x_3x_1+12x_1^2-6x_1,
\\
2x_3^3-2x_3x_1^2-x_5x_3+3x_3^2+x_5x_1+3x_3x_1-2x_3,\\
2x_5x_3^2-2x_5x_1^2-x_5x_3+3x_3^2+x_5x_1+3x_3x_1-2x_3.
\end{gather*}
The zero set of this radical ideal, excluding solutions in which any parameter is zero, and using the
equations $x_2 = x_1$ and $x_4 = x_3$, consists of the point $\big[ \tfrac13, \tfrac13, \tfrac13, \tfrac13, -\tfrac23, \tfrac13\big]$ and the family
\begin{equation}
\label{case1}
\Big[ x_1, x_1, x_3, x_3, x_3, x_1-1 \Big] \; \text{where} \; x_3^2+x_3-(x_1-1)^2=0 \; \text{and} \; x_1 \ne 0,1
\end{equation}

\noindent \emph{Case 2}. The original basis of 65 elements reduces to 29 elements;
the resulting basis has 25, 4 elements of degrees 4, 5 respectively, terms from 20 to 32, and
coefficients from $-2509$ to $5018$.
The radical has this Gr\"obner basis:
\begin{gather*}
(3x_1-1)(x_3-x_1),\;\;
3x_6x_1-3x_5x_1+2x_3+x_1,\;\;
(x_3-x_1)(x_3+x_1),
\\
(x_5-1)(x_3+x_1),\;\;
3x_6x_3+3x_5x_1-2x_3-x_1,\;\;
3x_5x_6-x_1+x_3,
\\
x_1(x_5-1)(3x_1-1),\;
9x_5^2x_1+3x_5x_1-5x_3-7x_1,\;
9x_6^3+18x_5x_1-9x_6-11x_3-7x_1.
\end{gather*}
The zero set of this radical ideal, excluding solutions in which any parameter is zero, and using the
equations $x_2 = -x_1$ and $x_4 = -x_3$, is as follows:
\begin{equation}
\label{case2}
[x_1,-x_1,x_3,-x_3,x_5,x_6]
\; = \;
\Big[ \tfrac13, -\tfrac13, -\tfrac13, \tfrac13, -\tfrac23, -\tfrac13 \Big], \;
\Big[ \tfrac13, -\tfrac13, -\tfrac13, \tfrac13,  \tfrac13,  \tfrac23 \Big]
\end{equation}

\noindent \emph{Case 3}: The results are very similar to those of case 2.
The radical of the ideal has the following Gr\"obner basis:
\begin{gather*}
 (3x_1-1)(x_2+x_1),\;\;
 3x_6x_1+3x_5x_1+2x_2-x_1,\;\;
 (x_2-x_1)(x_2+x_1),
 \\
 (x_5-1)(x_2-x_1),\;\;
 3x_6x_2+3x_5x_1+2x_2-x_1,\;\;
 3x_5x_6+x_1+x_2,
 \\
 x_1(x_5-1)(3x_1-1),\; 
 9x_5^2x_1+3x_5x_1+5x_2-7x_1,\; 
 9x_6^3-18x_5x_1-9x_6-11x_2+7x_1.
\end{gather*}
The zero set of this radical ideal, excluding solutions in which any parameter is zero, and using the
equations $x_3 = x_1$ and $x_4 = x_2$, is as follows:
\begin{equation}
\label{case3}
[x_1,x_2,x_1,x_2,x_5,x_6]
\; = \;
\Big[ \tfrac13, \tfrac13, \tfrac13, \tfrac13, -\tfrac23,  \tfrac13 \Big], \;
\Big[ \tfrac13, \tfrac13, \tfrac13, \tfrac13,  \tfrac13, -\tfrac23 \Big]
\end{equation}

\noindent \emph{Case 4}. 
The results are very similar to those of case 1.
The radical of the ideal has the following Gr\"obner basis:
\begin{gather*}
x_6-x_5-x_2+x_1-1,
\quad
(x_5+x_2)(2x_5+x_2-x_1+2),\\
2x_2^2x_1-2x_1^3-x_5x_2-x_2^2-x_5x_1-3x_2x_1+4x_1^2-2x_1,
\\
6x_5x_2x_1-6x_5x_1^2-6x_2x_1^2+6x_1^3+x_5x_2+x_2^2+5x_5x_1+11x_2x_1-12x_1^2+6x_1,
\\
2x_2^3-2x_2x_1^2-x_5x_2-3x_2^2-x_5x_1+3x_2x_1-2x_2,\\
2x_5x_2^2-2x_5x_1^2+x_5x_2+3x_2^2+x_5x_1-3x_2x_1+2x_2.
\end{gather*}
The zero set of this radical ideal, excluding solutions in which any parameter is zero, and using the
equations $x_3 = -x_1$ and $x_4 = -x_2$, consists of the point $\big[\tfrac13, -\tfrac13, -\tfrac13, \tfrac13, -\tfrac23, -\tfrac13\big]$ and the family
\begin{equation}
\label{case4}
\Big[ x_1, x_2, -x_1, -x_2, -x_2, 1-x_1 \Big], \;\; x_2^2-x_2-(x_1-1)^2=0, \;\; x_1 \ne 0,1
\end{equation}
To decide which (if any) of the solutions we found are regular,
we must compute Gr\"obner bases for the radicals of the lower determinantal ideals of the matrices
$B(\lambda)$, and then substitute the solutions into the Gr\"obner bases.

\subsubsection*{Lower ideals}

Here are the Gr\"obner bases for the radicals of the lower determinantal ideals:
\begin{alignat*}{1}
\lambda= 4\colon 
&
\quad
x_6+x_5-1, \quad
(x_4+x_2+1)(x_2+x_1), \quad
(x_2+x_1)(2x_5-x_3+x_2-2x_1),
\\
&\quad
x_4x_3+x_3^2+x_4x_1+x_3x_1-x_2-x_1, \\
&\quad
x_4^2-x_3^2+x_3x_2-x_2^2-2x_4x_1-x_3x_1-x_2x_1+x_4+x_3,
\\
&\quad
2x_5x_4+2x_5x_3+x_3x_2-x_2^2-2x_4x_1-x_3x_1-x_2x_1-2x_2-2x_1,
\\
&
\quad
(x_2+x_1)(x_3^2-x_3x_2+x_3x_1-x_2x_1-x_3-x_2-2x_1)
\\
\lambda= 31\colon
&
\quad
x_4+x_3+x_2+x_1, \quad
x_6x_1+x_5x_1-x_3, \quad
(x_2-x_1)(x_2+x_1), \\
&\quad
(x_3+x_1)(x_2+x_1), \quad
(x_5+1)(x_2+x_1), \quad
x_6x_2-x_5x_1+x_3, \\
&
\quad
(x_3-x_1)(x_3+x_1), \quad
(x_5-1)(x_3+x_1), \quad 
x_6x_3-x_5x_1+x_3, \\
&\quad
2x_5^2-4x_5x_1-x_3x_1-x_2x_1-2x_6+3x_3-x_2-2,
\\
&\quad
2x_6x_5+4x_5x_1+x_3x_1+x_2x_1-3x_3+x_2, \\
&\quad
2x_6^2-4x_5x_1-x_3x_1-x_2x_1+2x_6+3x_3-x_2\\
\lambda= 2^2\colon 
&
\quad
x_3-x_2, \qquad
x_4-x_1, \qquad
(x_2-x_1)(x_2+x_1), \quad
x_5x_2-x_6x_1-x_2,\\
& \quad
x_6x_2-x_5x_1+x_1,\quad
(x_6+x_5-1)(x_6-x_5+1), \\
&\quad 
(x_5-x_1-1)(x_5-x_1)(x_5+2x_1-1),
\\
&\quad
x_6x_5^2-3x_6x_1^2+2x_2x_1^2-x_6x_5+x_6x_1-2x_2x_1
\\
\lambda= 21^2\colon 
&
\quad
x_4-x_3-x_2+x_1, \quad
x_6x_1-x_5x_1-x_2, \quad
(x_2-x_1)(x_2+x_1), \\
&\quad
(x_3-x_1)(x_2-x_1),\quad
(x_5-1)(x_2-x_1), \quad
x_6x_2-x_5x_1-x_2, \\
&\quad
(x_3-x_1)(x_3+x_1), \quad
(x_5+1)(x_3-x_1), \quad
x_6x_3-x_5x_1-x_2, \\
&\quad
2x_5^2-4x_5x_1+x_3x_1+x_2x_1+2x_6+x_3-3x_2-2,\\
&\quad
2x_6x_5-4x_5x_1+x_3x_1+x_2x_1+x_3-3x_2, \\
&\quad
2x_6^2-4x_5x_1+x_3x_1+x_2x_1-2x_6+x_3-3x_2\\
\lambda = 1^4\colon
&
\quad
x_6-x_5+1, \qquad
x_4x_2-x_2^2-x_4x_1+x_2x_1-x_3+x_1, \\
&\quad
(x_4-x_3+1)(x_3-x_1),\quad
(x_3-x_1)(2x_5-x_3+x_2-2x_1),\\
& \quad
x_4^2-x_3^2+x_3x_2-x_2^2-2x_4x_1+x_3x_1+x_2x_1+x_4-x_2,
\\
&
\quad
2x_5x_4-x_3^2-2x_5x_2+x_3x_2-2x_4x_1+x_3x_1+x_2x_1+2x_3-2x_1,
\\
&
\quad
(x_3-x_1)(x_3x_2-x_2^2-x_3x_1+x_2x_1-x_3-x_2+2x_1)
\end{alignat*}

\subsubsection*{Comparison of upper and lower ideals}

Finally, we need to check if some of the parameter values \eqref{case1}--\eqref{case4} are common zeros for at least one of the Gr\"obner bases for the lower ideals.
For example, substituting the solution $\big[ \tfrac13, \tfrac13, \tfrac13, \tfrac13, -\tfrac23, \tfrac13\big]$ into the elements of the five
Gr\"obner bases produces the following lists of scalars:
\[
\begin{array}{ll}
\lambda = 4 &\quad  -\frac43, \; \frac{10}9, \; -\frac43, \; -\frac29, \; \frac29, \; -\frac83, \; -\frac89 \\[1mm]
\lambda = 31 &\quad  \frac43, \; -\frac49, \; 0, \; \frac49, \; \frac29, \; \frac23, \; 0, \; -\frac{10}9, \; \frac23, \; -\frac49, \; -\frac{16}9, \; \frac{20}9  \\[1mm]
\lambda = 2^2 &\quad  0, \; 0, \; 0, \; -\frac23, \; \frac23, \; -\frac83, \; -2, \; \frac29 \\[1mm]
\lambda = 21^2 &\quad 0, \; 0, \; 0, \; 0, \; 0, \; 0, \; 0, \; 0, \; 0, \; 0, \; 0, \; 0 \\[1mm]
\lambda = 1^4 &\quad  2, \; 0, \; 0, \; 0, \; 0, \; 0, \; 0 
\end{array}
\]
The five ideals are therefore
$( 1 )$, $( 1 )$, $( 1 )$, $( 0 )$, $( 1 )$
and so regularity fails since the fourth ideal is zero.
Similar calculations eliminate the other isolated points, and so it remains to check only the
one-parameter solutions \eqref{case1} and \eqref{case4}:
\begin{equation}
\label{last2}
\begin{array}{lrl}
\big[ \; x_1, \; x_1, \; x_3, \; x_3, \; x_3, \; x_1-1 \; \big] &\qquad x_3^2+x_3-(x_1-1)^2=0, \qquad x_1 \ne 0, 1
\\[2pt]
\big[ \; x_1, \; x_2, \; -x_1, \; -x_2, \; -x_2, \; 1-x_1 \; \big] &\qquad x_2^2-x_2-(x_1-1)^2=0, \qquad x_1 \ne 0, 1
\end{array}
\end{equation}
Each of these solutions is a sextuple depending on two parameters subject to one equation.  We substitute these solutions into the elements of the Gr\"obner bases of the radicals of the lower ideals, adjoin the equation relating the parameters to each of the five Gr\"obner bases, and solve the corresponding systems of equations. The union of all those solutions is precisely the set of values of parameters we must exclude.

An example will make this clear. Consider the first solution from \eqref{last2}. We substitute these values into the Gr\"obner basis for the radical of the lower determinantal ideal for $\lambda = 2^2$. The eight generators of that ideal become
\begin{gather*}
0,x_3-x_1, -x_1(x_1-x_3), x_1(x_1-x_3), (x_1-2+x_3)(x_1-x_3), \\ -( x_1-1)( x_1-1+x_3)( x_1-x_3), ( x_3+2x_1-1)( -x_3+1+x_1)(x_1-x_3),
\end{gather*}
and we see that all these are equal to zero only when $x_1=x_3$.
Taking into account the equation $x_3^2+x_3-(x_1-1)^2=0$, we see that $x_1=x_3=\frac13$.
Doing these for all determinantal ideals, we find that the points that have to be removed from the first family are $(1,-1)$ and $(\frac13,\frac13)$, and from the second family, $(1,1)$ and $(\frac13,-\frac13)$.
(In fact, the first point in each pair has already been removed, since we assume $x_1\ne1$.)

The formulation of the result becomes a little bit more elegant if we replace $x_3$ by $-x_3$ in the first family.

\begin{proposition}[Summary for six nonzero parameters]\label{prop:6param}
	When all parameters in Relation \eqref{LRrelation} are nonzero,  there are two one-dimensional families of solutions which are regular in arity $4$:
\begin{align*}
	(a_1a_2)a_3=& \; x_1\big[a_1(a_2a_3) + a_1(a_3a_2)  + a_3(a_2a_1)\big] 
-x_3 \big[ a_2(a_1a_3) + a_2(a_3a_1) + a_3(a_1a_2)\big] - a_3(a_2a_1),\\
	(a_1a_2)a_3=&\; x_1\big[a_1(a_2a_3) - a_2(a_1a_3)  - a_3(a_2a_1)\big] 
 +x_2 \big[ a_1(a_3a_2)- a_2(a_3a_1) - a_3(a_1a_2)\big] + a_3(a_2a_1),
\end{align*}	
	where both $(x_1,x_2)$ and $(x_1,x_3)$ belong to the hyperbola $y^2-y-(x-1)^2=0$ with five excluded points: $(1,0)$, $(1,1)$, $(\frac13,-\frac13)$, and $(0,\phi)$ for both roots $\phi$ of the polynomial $x^2-x-1$.
\end{proposition}

\subsection{Statement of the main technical result}

After noticing that the excluded points $(0,\phi)$ in the last statement are precisely the points with four nonzero parameters that we found previously, and the excluded point $(1,0)$ corresponds to the Zinbiel operad in the first case and to the Leibniz operad in the second case,  we see that Propositions \ref{prop:1param}--\ref{prop:6param} lead to the following conclusion.

\begin{theorem}\label{th:classification}
The parametrized one-relation operads with the regular module in arity $4$ are precisely the operads from the following list:
\begin{itemize}
\item[(i)] $(a_1a_2)a_3 = s\, a_3(a_1a_2)$,
\item[(ii)] $(a_1a_2)a_3=a_1(a_2a_3) + s\big[a_1(a_3a_2) - a_2(a_1a_3) + a_2(a_3a_1) -a_3(a_1a_2)\big]$,
\item[(iii)] $(a_1a_2)a_3= u\big[a_1(a_2a_3) + a_1(a_3a_2)  + a_3(a_2a_1)\big]  -v \big[ a_2(a_1a_3) + a_2(a_3a_1) + a_3(a_1a_2)\big] - a_3(a_2a_1)$,
\item[(iv)] $(a_1a_2)a_3=u\big[a_1(a_2a_3) - a_2(a_1a_3)  - a_3(a_2a_1)\big] +v \big[ a_1(a_3a_2)- a_2(a_3a_1) - a_3(a_1a_2)\big] + a_3(a_2a_1)$,	
\end{itemize}
where in the case (i) we require $s\ne\pm1$, in the case (ii) we require $s\ne -1$, and in both cases (iii) and (iv) the point $(u,v)$ belongs to the hyperbola $y^2-y-(x-1)^2=0$ with the points $(1,1)$ and $(\frac13,-\frac13)$ excluded.
\end{theorem}

\section{Classification theorem}\label{actualclassification}

In this section, we prove the following classification result which is the main result of this paper. 

\begin{theorem}\label{th:isoclassification}\leavevmode
\begin{itemize}
\item[(i)] Over any field~$\mathbb{F}$ of characteristic $0$, each regular parametrized one-relation operad is one of the operads of Theorem \ref{th:classification}.
\item[(ii)] Over an algebraically closed field $\mathbb{F}$ of characteristic $0$, every regular parametrized one-relation operad
is isomorphic to one of the following five operads: the left-nilpotent operad defined by the identity
$((a_1a_2)a_3)=0$, the associative operad, the Leibniz operad $\Leib$, the dual Leibniz (Zinbiel) operad
$\Zinb$, and the Poisson operad.
\end{itemize}
\end{theorem}

\begin{proof}
We shall go through the list of Theorem \ref{th:classification} and establish that each of the operads which have the regular module in arity $4$ is in fact regular and isomorphic to one of the five operads listed above; we will use the following observation. For each $t\in\mathbb{F}$, one has the following endomorphism $\phi_t$ of the space of generators $\calT(2)$ of the free operad:
	\[
	\phi_t(a_1a_2)=a_1a_2+t a_2a_1,
	\qquad
	\phi_t(a_2a_1)=a_2a_1+t a_1a_2.
	\]
	This endomorphism commutes with the symmetric group action, and is invertible if and only if $t\ne\pm1$.
	(This change of basis was studied by Livernet and Loday in the context of relating the Poisson operad to
	the associative operad \cite{LL,MR1}.  See also a similar change of basis in the space of operations in
	the work of Albert \cite[\S V]{Al} in the context of power-associative and quasiassociative rings.)
	It extends to a well-defined endomorphism of the free operad $\calT$.
	We can replace Relation \eqref{LRrelation} by its image under this endomorphism, which is one of relations
	of the general type \eqref{relation2}.  Recall from Formula \eqref{rowspace1} that the general space of
	relations is spanned by the rows of the $6\times 12$ matrix $N = [ W \mid X ]$.
	If $\det W \ne 0$, then, according to Lemma \ref{leftminor}, there exists an equivalent relation of the
	type \eqref{LRrelation}.  Overall, this allows us to find, for each regular parametrized one-relation
	operad, a one-parameter family of regular parametrized one-relation operads which are isomorphic to it;
	the set of parameters is precisely the set of all $t$ for which $\det W \ne 0$.
	We now make this outlined strategy more precise.
	
Note that the endomorphism of $\calT(3)$ induced by $\phi_t$ is given by the matrix
	\[
		A(t) =
		\begin{bmatrix}
		1&.&t&.&.&.&.&.&.&.&t&t^2\\
		.&1&.&.&t&.&.&.&t&t^2&.&.\\
		t&.&1&.&.&.&.&.&.&.&t^2&t\\
		.&.&.&1&.&t&t&t^2&.&.&.&.\\
		.&t&.&.&1&.&.&.&t^2&t&.&.\\
		.&.&.&t&.&1&t^2&t&.&.&.&.\\
		.&.&.&t&.&t^2&1&t&.&.&.&.\\
		.&.&.&t^2&.&t&t&1&.&.&.&.\\
		.&t&.&.&t^2&.&.&.&1&t&.&.\\
		.&t^2&.&.&t&.&.&.&t&1&.&.\\
		t&.&t^2&.&.&.&.&.&.&.&1&t\\
		t^2&.&t&.&.&.&.&.&.&.&t&1
		\end{bmatrix}
		\]
(zeros are replaced by dots for readability). This can be established by a direct calculation. For instance,
	\begin{align*}
		\phi_t((a_1a_2)a_3)& \; =((a_1a_2+t\,a_2a_1)a_3+t\,(a_3(a_1a_2+t\,a_2a_1))\\
		&\; =(a_1a_2)a_3+t\,(a_2a_1)a_3+t\,a_3(a_1a_2)+t^2\,a_3(a_2a_1),	
	\end{align*}
		which precisely corresponds to the first column of the matrix $A(t)$.
	
	Suppose that $N_0$ is the $6\times 12$ matrix whose rows form the $S_3$-orbit of some relation of the type \eqref{LRrelation}. The change of basis we introduced amounts to multiplying $N_0$ by $A(t)$ on the right. We let $N(t)=N_0A(t)=[W(t)\mid Y(t)]$, where $W(t)$ and $Y(t)$ are $6\times6$ matrices with entries in $\mathbb{F}[t,x_1,\ldots,x_6]$. The module of quadratic relations generated by the rows of this matrix contains a relation of type \eqref{LRrelation} if and only if $\det W(t)\ne 0$, and that in this case the matrix $\widetilde{N}(t)=W(t)^{-1}N(t)$ encodes that relation.
	
	We are now ready to investigate the isomorphism classes. We start with the parametric family \[ (a_1a_2)a_3=s a_3(a_1a_2) ,
	\quad  s\ne\pm 1 . \] We have $\det W(t)=(1-t)^3(t+1)^3(1-st)^6$. The change of basis given by $A(t)$ results in the following change of parametrization:
	\[
	\tilde{s}=\frac{t-s}{st-1}.
	\]
	Clearly, if we put $t = s$, then $\det W(t)\ne 0$, and $\tilde{s}=0$. Therefore, each operad of this family is isomorphic to the left-nilpotent operad.
	
	Next, we consider the parametric family
	\[
	(a_1a_2)a_3=a_1(a_2a_3) + s\big[a_1(a_3a_2) - a_2(a_1a_3) + a_2(a_3a_1) -a_3(a_1a_2)\big] ,
	\]
	where $s\ne -1$. We have $\det W(t)=(1-t)^5(t+1)^3(3st+t^2+t+1)^2$. The change of basis given by $A(t)$ results in the following change of parametrization:
	\[
	\tilde{s}=\frac{st^2 - st + s + t}{3st + t^2 + t + 1}.
	\]
	The resultant with respect to $t$ of  the product of irreducible factors of $\det W(t)$ and the numerator of $\tilde{s}$ is, as one can check by an immediate computation, equal to
	$
	(s+1)^3(3s-1)^3.
	$	Therefore, for each point $s\ne-1,\frac13$, it is possible find a value of $t$ for which $\det W(t)\ne0$, and $\tilde{s}=0$. For such $t$, we see that there is a change of basis that makes $\tilde{s}=0$, so each operad of this family except for the operad for $s=\frac13$ is isomorphic to the associative operad. The operad for $s=\frac13$ is a fixed point for all changes of basis; it is the one-operation presentation of the operad of Poisson algebras \cite{LL,MR1}.
	
	Finally, we consider the parametric families
	\begin{align*}
	(a_1a_2)a_3=&\; u\big[a_1(a_2a_3) + a_1(a_3a_2)  + a_3(a_2a_1)\big] 
-v \big[ a_2(a_1a_3) + a_2(a_3a_1) + a_3(a_1a_2)\big] - a_3(a_2a_1) ,\\
	(a_1a_2)a_3=&\; u\big[a_1(a_2a_3) - a_2(a_1a_3)  - a_3(a_2a_1)\big] 
+v \big[ a_1(a_3a_2)- a_2(a_3a_1) - a_3(a_1a_2)\big] + a_3(a_2a_1),
	\end{align*}
	where the parameters $u$ and $v$ are related by the equation $v^2-v-(u-1)^2=0$, and $(u,v)\ne(1,1), (\frac13,-\frac13)$. We have $\det W(t)=(1-t)^3(t+1)^5(ut + vt - t + 1)^3(1+t-3ut + 3vt)$. The change of basis given by $A(t)$ is the following change of parametrization:
	\begin{align*}
	\tilde{u}= \frac{2u^2t^2 + u^2t - ut^2 - u - 2v^2t^2 - v^2t - 2vt}{3u^2t^2 - 4ut^2 + 2ut - 3v^2t^2 + 2vt^2 - 4vt + t^2 - 1} , \\[1mm]
	\tilde{v}=  \frac{u^2t^2 + 2u^2t - 2ut - v^2t^2 -  2v^2t  - vt^2 - v}{3u^2t^2 - 4ut^2 + 2ut - 3v^2t^2 + 2vt^2 - 4vt + t^2 - 1} .
	\end{align*}
	The resultant with respect to $t$ of the product of irreducible factors of $\det W(t)$ and the numerator of $\tilde{v}$ is
	\[
	(u-v)^2(u+v)^2(u+v-2)^2(2u-v-1)^2(3u-3v-2)^2.
	\]
	This polynomial has common roots with $v^2-v-(u-1)^2=0$ if and only if $(u,v)=(1,1)$ or $(u,v)=(\frac13,-\frac13)$, which are precisely the points we excluded. Therefore, for each operad in each of the two families, it is possible find a value of $t$ for which $\det W(t)\ne0$, and $\tilde{v}=0$. For such $t$, we see that there is a change of basis that makes $\tilde{v}=0$, which in turn forces $\tilde{u}=1$. This proves that each operad of the first family is isomorphic to the Zinbiel operad, and each operad of the second family is isomorphic to the Leibniz operad.
\end{proof}

\section{Further directions}
\label{sectionirregular}

\subsection{Further questions about the cubic relation matrix}

It would be interesting to extend the nilpotency result of Section \ref{cubicrelationmatrixsection} and
classify all parametrized one-relation operads which are nilpotent.
There are two somewhat natural questions one may ask here.

\begin{problem}
	\label{determinantMproblem}
	Determine explicitly the factorization of the determinant of the cubic relation matrix $M$ into the product of irreducible polynomials in
	$\mathbb{F}[x_1,\dots,x_6]$. Use this to determine explicitly all parameter values $a_1,\dots,a_6 \in \mathbb{F}$ for which the operad
	$\calO_\mathbf{a}$ is nilpotent of index $3$: these values form the complement $\mathbb{F}^6 \setminus V(\det(M))$.
\end{problem}

\begin{problem}
	For every $d \ge 3$, determine explicitly the set $\calN_d \subseteq \mathbb{F}^6$ of all parameter values
	$a_1,\dots,a_6$ for which the operad $\calO_\mathbf{a}$ is nilpotent of index $d$.
	For these values, we have $\calJ(d) \ne \calT(d)$ and $\calJ(d+1) = \calT(d+1)$.
	We have already seen in Theorem \ref{th:niltheorem} that the set $\calN_3$ is a Zariski open subset of
	$\mathbb{F}^6$.
\end{problem}

We have been able to use representation theory in order to avoid dealing with the determinantal ideals of the cubic relation matrix $M$, or, equivalently, of the block $B$ of its partial Smith normal form. Understanding the structure of those ideals remains an open problem.

\begin{problem}
	For $r = 1$, the reduced Gr\"obner bases for the first determinantal ideal $DI_1(B)$ and its radical were
	presented in Lemma \ref{basicMinfo}.
	For $2 \le r \le 36$, an open problem (probably rather hard, at least computationally) is to determine the reduced Gr\"obner bases for the
	$r$-th determinantal ideal $DI_r(B)$ and its radical.
	For $r = 36$, the determinantal ideal $DI_{36}(B)$ is the principal ideal generated by $\det(B)$, and
	by Algorithm \ref{partialsmithform} we know that $\det(B) = \pm \det(M)$, so this case overlaps with Problem \ref{determinantMproblem}.
\end{problem}

\subsubsection{Rank distribution for relations with small coefficients}\label{sec:rankdistr}

Let us conclude this subsection with some experimental data that sheds some light on the rank distribution
for the cubic relation matrix as a function of the parameter values.
We consider the $729$ relations \eqref{LRrelation} with coefficients in $\{ \, 0, \, \pm 1 \, \}$, and
we partition this set by the number $q$ of nonzero coefficients.
In each case, we substitute the parameter values into $B$ and compute $r = \mathrm{rank}(B)$, recalling that
$\rank(M) = 84 + \rank(B)$.
In the following table, the rows are indexed by $q$ and the columns by $r$.
The $(q,r)$ entry is the number of relations for which $x_1,\dots,x_6 \in \{ \, 0, \, \pm 1 \, \}$
and $\big|\{ \, i \mid x_i \in \{\pm1\} \, \}\big| = q$ and $\mathrm{rank}(B) = r$ where $0 \le q \le 6$ and
$0 \le r \le 36$ (as above, zeros are replaced by dots for readability):
\smallskip
\begin{center}
	$
	\begin{array}
	{c|c
		cccccccccccc
		ccccccc}
	{}_q\backslash^r &
	0 &&&&&& 6 &&&&&& 12 &&&&&& 18&19
	\\
	\midrule
	0 &  . &
	. &  . &  . &  . &  . &  . &  . &  . &  . &  . &  . &  1 &
	. &  . &  . &  . &  . &  . &.\\
	1 &  2 &
	. &  . &  . &  . &  . &  2 &  . &  . &  . &  . &  . &  1 &
	. &  . &  . &  . &  . &  . &.\\		
	2 &  . &
	. &  . &  . &  . &  . &  . &  . &  . &  . &  . &  . &  2 &
	. &  . &  . &  2 &  .&.&.\\		
	3 &  . &
	. &  . &  . &  . &  . &  . &  . &  . &  . &  . &  . &  . &
	. &  2 &  2 &  . &  8 &  3 &.\\		
	4 &  . &
		. &  . &  . &  . &  . &  . &  . &  . &  . &  . &  . &  . &
		. &  . &  . &  . &  . & .&.\\	
   5 &  . &
		. &  . &  . &  2 &  . &  2 &  . &  . &  . &  2 &  . &  3 &
		4 &  . &  . &  . &  . & 10&.\\	
   6 &  . &
		. &  . &  . &  . &  . &  . &  . &  . &  . &  . &  . &  . &
		. &  . &  . &  . &  . &  . &.\\	
	\midrule
	\sum &  2 &
	. &  . &  . &  2 &  . &  4 &  . &  . &  . &  2 &  . &   7 &
	4 &  2 &  2 &  2 &  8 & 13 &.		
	\end{array}$

\medskip 
		
	$\begin{array}
	{c|c
		ccccc
		cccccccccccc}
	{}_q\backslash^r &
	&&&& 24 &&&&&& 30 &&&&&& 36
	\\
	\midrule
    0&   . &  . &  . &  . &  . &
	. &  . &  . &  . &  . &  . &  . &  . &  . &  . &  . &  .
	\\
     1&   . &  . &  . &  . &  6 &
	. &  . &  . &  . &  . &  . &  . &  . &  . &  . &  . &  1
	\\
    2  &   3 &  . &  . &  2 &  5 &
	. &  . &  . &  6 &  . &  8 &  2 &  4 &  2 &  2 &  2 & 20
	\\
    3&    5 & 18 &  . &  . & 12 &
	. &  2 &  8 &  1 & 16 & 10 &  4 & 12 & 28 &  7 &  8 & 14
	\\
    4&     2 &  . &  8 &  4 &  4 &
	2 &  . &  . & 26 & 12 & 12 &  8 & 20 & 14 & 12 & 18 & 98
	\\
    5 &    . &  8 &  2 &  . & 20 &
	. &  1 &  6 &  2 &  8 & 14 &  4 &  9 & 38 &  4 & 12 & 41
	\\
    6&    . &  . &  . &  . &  8 &
	. &  4 &  . &  . &  . & 20 &  . &  8 &  . &  . &  . & 24
	\\
		\midrule
   \sum&   10 & 26 & 10 &  6 &  55 &
	2 &  7 & 14 & 35 & 36 & 64 & 18 & 53 & 82 & 25 & 40 & 198
	\end{array}
	$
\end{center}
\smallskip
From column $36$ we see that $198/729 \cong 27.16\%$ of these operads are nilpotent of index $3$.
Regularity implies $\mathrm{rank}(B) = 12$ but not conversely; column $12$ indicates that there are respectively
$1, 1, 2, 3$ relations for $q = 0, 1, 2, 5$ with $\mathrm{rank}(M) = 96$.
In these seven cases, the parameter values are the rows of the following matrix, and the last column gives the
multiplicities for the $S_4$-action on the nullspace of $M$:
\smallskip
\begin{center}
	$\begin{array}{|rrrrrr|c}
	0  &  0  &  0  &  0  &  0  &  0  &\qquad\qquad  [1,3,2,3,1] \\
	1  &  0  &  0  &  0  &  0  &  0  &\qquad\qquad  [1,3,2,3,1] \\
	1  &  1  &  0  &  0  &  0  &  0  &\qquad\qquad  [1,3,2,3,1] \\
	1  &  0  & -1  &  0  &  0  &  0  &\qquad\qquad  [1,3,2,3,1] \\
	1  &  1  & -1  &  1  & -1  &  0  &\qquad\qquad  [1,3,2,3,1] \\
	-1  &  1  &  1  & -1  &  0  &  1  &\qquad\qquad  [2,3,3,2,1] \\
	-1  & -1  & -1  & -1  &  0  & -1  &\qquad\qquad  [1,2,3,3,2]
	\end{array}$
\end{center}
\smallskip

\subsection{Koszul operads with one relation} \label{KOOR}

A question of Loday that we mentioned in Introduction still remains open:

\begin{problem}
	Which of the parametrized one-relation operads $\calO_\mathbf{x}$ are Koszul?
\end{problem}

Theorem \ref{th:isoclassification} of course implies that all regular parametrized one-relation
operads are Koszul, while Theorem \ref{th:niltheorem} easily implies that generic parametrized
one-relation operads are not Koszul.
The Hilbert series of an index $3$ nilpotent parametrized one-relation operad is
$f(t) = t + t^2 + t^3$;
the modified inverse series has negative coefficients:
\[
-f^{\langle-1\rangle}(-t) = t + t^2 + t^3 - 4t^5 - 14 t^6 - 30 t^7 - 33 t^8 + 55 t^9 + O(x^{10}).
\]
The Koszulness criterion of Ginzburg and Kapranov \cite{GK,LV} instantly implies that such an operad
cannot be Koszul.
Moreover, inspecting the list of $729$ parametrized one-relation operads with coefficients in
$\{ \, 0, \, \pm 1 \, \}$ from \S\ref{sec:rankdistr}, we discover that most of those operads are
not Koszul either because the modified inverse of the Hilbert series has negative coefficients, or
because the Hilbert series of the operad is not equal to the inverse of the modified Hilbert series of the Koszul dual operad
(which is, as we know, isomorphic to a parametrized one-relation operad). Among those $729$ operads, there
are just six irregular cases where the Koszulness cannot be disproved using the Ginzburg--Kapranov criterion.
Four of those,
\[
(a_1a_2)a_3=\pm a_1(a_2a_3) \qquad\text{and}\qquad  (a_1a_2)a_3=\pm a_1(a_3a_2),
\]
are Koszul and in fact have quadratic Gr\"obner bases for the (weighted) \texttt{pathdeglex} ordering \cite{DK}.
(We encountered two of those operads in Lemma \ref{basicMinfo}; notably, the corresponding $S_4$-modules
both have dimension 36 but are not isomorphic: the multiplicities are $[2,4,4,4,2]$ for the relation $(a_1a_2)a_3 = a_3(a_1a_2)$ and
$[1,5,2,5,1]$ for the relation $(a_1a_2)a_3 = -a_3(a_1a_2)$). Two remaining operads for which the Koszulness
remains an open question are
\[
(a_1a_2)a_3=\pm\big[a_1(a_2a_3)-a_1(a_3a_2)+a_2(a_1a_3)-a_2(a_3a_1)\big] + a_3(a_1a_2),
\]
one of which we saw as an excluded point of an otherwise regular family of parametrized one-relation operads in Remark~\ref{rem:5paramexceptional}.

\appendix

\section{Verification of results in \texttt{Magma}}

Our computer algebra system of choice for this project was \texttt{Maple}. Computations above use various tricks and divide-and-conquer methods designed to avoid using asking \texttt{Maple} to compute the radical of an ideal: at least in \texttt{Maple 18} which we were using at the crucial stage of this project the implementation of radical computation seemed to have some bugs (which seem to have been fixed in \texttt{Maple 2016}). As an independent verification, we used the \texttt{RadicalDecomposition} function of \texttt{Magma} \cite{Magma}, which appears to be extremely efficient even in the free online calculator
\begin{center}
\url{http://magma.maths.usyd.edu.au/calc/}
\end{center}
which limits the input to 50Kb and the calculation time to 120 seconds. We fed into that calculator the respective blocks $B(\lambda)$ \cite{BDadd} obtained by partial reduction of representation matrices (which were obtained through simple linear algebra over the rational field by a direct computation not involving any complicated \texttt{Maple} functions, and thus represented the ``fool-proof'' part of the computation), and requested the calculator to compute the following:
\begin{itemize}
\item all the upper determinantal ideals $DI_{r+1}(B(\lambda))$;
\item all the lower determinantal ideals $DI_r(B(\lambda))$;
\item the prime decomposition of the radical of the sum of the upper determinantal ideals;
\item the prime decompositions of the radicals of the five ideals obtained as sums of upper ideals for four out of five $\lambda$ and the lower ideal for the remaining choice of $\lambda$. 
\end{itemize}
(The simple \texttt{Magma} script that we used is given in the online addendum \cite{BDadd}.) This computation took less than five seconds, and the result obtained was as follows. 

\begin{theorem}
The zero set of the sum of the upper ideal has ten irreducible components:
\begin{gather}
\big\{[1-x_6,-x_5,x_6-1,x_5,x_5,x_6] \colon x_6^2=x_5^2+x_5 \big\} , \label{mcase1}\\
\big\{[1+x_6,1+x_6,x_5,x_5,x_5,x_6] \colon x_6^2=x_5^2+x_5\big\} , \label{mcase2}\\
\big\{[-x_4,-x_4,x_4,x_4,-1,0]\big\} ,\label{mcase3}\\
\big\{[-x_4,x_4,-x_4,x_4,1,0\big\} ,\label{mcase4}\\
\big\{[1,-x_5,x_5,-x_5,x_5,0]\big\} , \label{mcase5}\\
\big\{[0,0,0,0,x_5,0]\big\} ,\label{mcase6}\\
\big\{\left[\tfrac13,-\tfrac13,-\tfrac13,\tfrac13,-\tfrac23,-\tfrac13\right]\big\} ,\label{mcase7}\\
\big\{\left[\tfrac13,\tfrac13,\tfrac13,\tfrac13,\tfrac13,-\tfrac23,\tfrac13\right]\big\} ,\label{mcase8}\\
\big\{[0,0,0,0,0,-1]\big\} , \label{mcase9}\\
\big\{[0,0,0,0,0,1]\big\} . \label{mcase10}
\end{gather} 
The zero sets of the five ideals obtained as sums of upper ideals for four out of five $\lambda$ and the lower ideal for the remaining choice of $\lambda$ are as follows.
\begin{itemize}
\item for $\lambda=4$
\begin{gather}
\big\{\left[-x_4,x_4,-x_4,x_4,1,0\right]\big\} ,\label{bad4:1}\\
\big\{\left[\tfrac13,-\tfrac13,-\tfrac13,\tfrac13,\tfrac13,\tfrac23\right]\big\} ,\label{bad4:2}\\
\big\{[0,0,0,0,0,1]\big\} . \label{bad4:3}
\end{gather}
\item for $\lambda=31$
\begin{gather}
\big\{[-x_4,-x_4,x_4,x_4,-1,0]\big\} ,\label{bad31:1}\\
\big\{[-x_4,x_4,-x_4,x_4,1,0\big\} ,\label{bad31:2}\\
\big\{\left[\tfrac13,-\tfrac13,-\tfrac13,\tfrac13,-\tfrac23,-\tfrac13\right]\big\} ,\label{bad31:3}\\
\big\{[0,0,0,0,0,-1]\big\} . \label{bad31:4}
\end{gather}
\item for $\lambda=2^2$
\begin{gather}
\big\{\left[\tfrac13,-\tfrac13,-\tfrac13,\tfrac13,\tfrac13,\tfrac23\right]\big\} ,\label{bad22:1}\\
\big\{\left[\tfrac13,\tfrac13,\tfrac13,\tfrac13,\tfrac13,-\tfrac23\right]\big\} ,\label{bad22:2}\\
\big\{[0,0,0,0,1,0]\big\} ,\label{bad22:3}\\
\big\{[0,0,0,0,-1,0]\big\} ,\label{bad22:4}\\
\big\{[0,0,0,0,0,1]\big\} . \label{bad22:5}
\end{gather}
\item for $\lambda=21^2$
\begin{gather}
\big\{[-x_4,-x_4,x_4,x_4,-1,0]\big\} ,\label{bad211:1}\\
\big\{[-x_4,x_4,-x_4,x_4,1,0\big\} ,\label{bad211:2}\\
\big\{\left[\tfrac13,\tfrac13,\tfrac13,\tfrac13,\tfrac13,-\tfrac23,\tfrac13\right]\big\} ,\label{bad211:3}\\
\big\{[0,0,0,0,0,1]\big\} . \label{bad211:4}
\end{gather}
\item for $\lambda=1^4$
\begin{gather}
\big\{[-x_4,x_4,-x_4,x_4,1,0\big\} ,\label{bad1111:1}\\
\big\{\left[\tfrac13,\tfrac13,\tfrac13,\tfrac13,\tfrac13,\tfrac13,-\tfrac23\right]\big\} ,\label{bad1111:2}\\
\big\{[0,0,0,0,0,-1]\big\} . \label{bad1111:3}
\end{gather}
\end{itemize}
\end{theorem}

The answer to our problem is obtained by removing from the first zero set the union of the remaining ones. First, we note the following:
\begin{itemize}
\item the component \eqref{mcase3} appears among the excluded ones as \eqref{bad31:1} and \eqref{bad211:1},
\item the component \eqref{mcase4} appears among the excluded ones as \eqref{bad4:1}, \eqref{bad31:2}, \eqref{bad211:2}, and \eqref{bad1111:1}, 
\item the component \eqref{mcase7} appears among the excluded ones as \eqref{bad31:3},
\item the component \eqref{mcase8} appears among the excluded ones as \eqref{bad211:3},
\item the component \eqref{mcase9} appears among the excluded ones as \eqref{bad31:4} and \eqref{bad1111:3}, 
\item the component \eqref{mcase10} appears  among the excluded ones as \eqref{bad4:3}, \eqref{bad22:5}, and \eqref{bad211:4}.
\end{itemize}

This means that we just need to examine which points are to be removed from the components \eqref{mcase1}, \eqref{mcase2}, \eqref{mcase5}, and \eqref{mcase6}. By a more careful inspection, we determine the following: 
\begin{itemize}
\item for the component \eqref{mcase1}, there are two points to be removed: the point corresponding to $x_5=\tfrac13$, $x_6=\tfrac23$ (it is the excluded component \eqref{bad4:2}, same as \eqref{bad22:1}) and the point corresponding to $x_5=-1$, $x_6=0$ (it corresponds to $x_4=-1$ in the excluded component \eqref{mcase3}),
\item for the component \eqref{mcase2}, there are two points to be removed: the point corresponding to $x_5=\tfrac13$, $x_6=-\tfrac23$ (it is the excluded component \eqref{bad22:2}, same as \eqref{bad1111:2}) and the point corresponding to $x_5=-1$, $x_6=0$ (it corresponds to $x_4=-1$ in the excluded component \eqref{mcase3}),
\item for the component \eqref{mcase5}, there is one point to be removed: the point corresponding to $x_5=1$ (it corresponds to $x_4=-1$ in the excluded component \eqref{mcase4})
\item for the component \eqref{mcase6}, there are two points to be removed: the point corresponding to $x_5=1$ (it is the excluded component \eqref{bad22:3}) and the point corresponding to $x_5=-1$ (it is the excluded component \eqref{bad22:4}).
\end{itemize}
By a direct inspection, this coincides with the set obtained in Theorem \ref{th:classification}, which completes the verification.

\end{document}